\documentclass[3p]{elsarticle}

\usepackage{amsfonts}
\usepackage{amsmath}
\usepackage{amssymb}
\usepackage{amsthm}

\usepackage{graphicx}
\usepackage{float}
\restylefloat{figure}

\journal{Computers \& Mathematics with Applications}

\usepackage{epic,eepic}

\newtheorem{theorem}{Theorem}[section]
\newtheorem{corollary}[theorem]{Corollary}
\newtheorem{lemma}[theorem]{Lemma}
\newtheorem{proposition}[theorem]{Proposition}

\newtheorem{definition}[theorem]{Definition}
\theoremstyle{remark}
\newtheorem{example}[theorem]{Example}
\newtheorem{remark}[theorem]{Remark}

\numberwithin{equation}{section}

\DeclareMathOperator{\ran}{range}

\def\RR{\mathbb R}
\def\CC{\mathbb C}
\def\NN{\mathbb N}
\def\ZZ{\mathbb Z}
\def\TT{\mathbb T}
\def\DD{\mathbb D}

\def\wt{\widetilde}
\def\wh{\widehat}

\def\ll{\ell}

\def\ep{\varepsilon}
\def\la{\lambda}
\def\vphi{\varphi}
\def\ga{\gamma}
\def\De{\Delta}

\def\X{\mathcal{X}}
\def\B{\mathcal{B}}
\def\Bo{\mathcal{B}_0}
\def\L{\mathcal{L}}
\def\W{\mathcal{W}}
\def\U{\mathcal{U}}
\def\Ssp{\mathcal{S}}

\def\A{\mathcal{A}}
\def\M{\mathcal{M}}

\def\ShB{S_{B,\mathrm{back}}}
\def\ShEl{S_{\mathrm{forw}}}
\def\r{\mathrm{r}}
\def\IS{\Gamma}
\def\IO{L}
\def\T{\mathrm{T}}
\def\Nper{N}
\def\c{\mathrm{c}}

\def\t{\mathrm{t}}

\def\Ufd{\mathcal{V}}
\def\E{\mathfrak{E}}
\def\N{\mathfrak{N}}
\def\IOfd{\mathfrak{L}}
\def\Oz{\mathit{0}}

\begin{document}

\begin{frontmatter}

\title{Structured stability radii and exponential stability tests for Volterra difference systems\tnoteref{tfn1,tfn2}}

\tnotetext[tfn1]{The authors were partially supported by the Pacific Institute for
Mathematical Sciences and by the first author's NSERC research grant.}
\tnotetext[tfn2]{The second author is grateful to the University of Calgary for hospitality.}

\author[eb]{E.~Braverman}
\ead{maelena@math.ucalgary.ca}
\address[eb]{Department of Mathematics and Statistics, University of Calgary,
2500 University Drive N. W.,\\ Calgary, AB T2N 1N4, Canada.}
\author[ik]{I.~Karabash\corref{cor1}}
\address[ik]{Institute of Applied Mathematics and Mechanics,
R. Luxemburg str. 74, \\ Donetsk, 83114, Ukraine}


\begin{abstract}
Uniform exponential (UE) stability of linear difference equations with infinite delay is studied
using the notions of a stability radius and a phase space. The state space $\X$ is supposed to be an abstract Banach space.
We work both with  non-fading phase spaces $c_0 (\ZZ^-, \X)$ and $\ll^\infty (\ZZ^-, \X)$
and with exponentially fading phase spaces of the $\ll^p$ and $c_0$ types.
For equations of the convolution type, several criteria of UE stability are obtained in terms of the Z-transform
$\wh K (\zeta)$ of the convolution kernel $K (\cdot)$, in terms of the input-state operator and of the resolvent (fundamental)
matrix. These criteria do not impose additional positivity or compactness assumptions on coefficients $K(j)$.
Time-varying (non-convolution) difference equations are studied via structured
UE stability radii $\r_\t$ of convolution equations. These radii correspond to a feedback scheme with delayed output and
time-varying disturbances. We also consider stability radii $\r_\c$ associated with a time-invariant disturbance operator,
unstructured stability radii, and stability radii corresponding to delayed feedback.
For all these types of stability radii two-sided estimates are obtained. The estimates from above are given
in terms of the Z-transform $\wh K (\zeta)$, the estimate from below via
the norm of the input-output operator. These estimates turn into explicit formulae if the
state space $\X$ is Hilbert or if disturbances are time-invariant.
The results on stability radii are applied to obtain various exponential stability tests
for non-convolution equations. Several examples are provided.

\end{abstract}
\end{frontmatter}

\section{Introduction}

The aim of the paper is to find or, in more involved cases, to estimate exponential stability radii for linear
convolution difference systems  with infinite delay
\begin{equation}
x(n+1) = \sum_{j=0}^{+\infty} K(j) x (n-j) ,  \quad n \geq 0, \label{e Keq intro}
\end{equation}
and then to apply obtained results to
the study of exponential stability of the Volterra difference system
\begin{equation}
x(n+1) = \sum_{j=0}^{+\infty} Q (n,j) x (n-j) ,  \quad n \geq 0, \label{e Veq}
\end{equation}
with time-varying (i.e., $n$-depending) coefficients $Q(n,j)$.
Here $x(\cdot)$ is a discrete function from $\ZZ$ to a
(complex) Banach space $\X$. $\X$ is called \emph{the state space}. The coefficients
$Q (n,j)$ belong to the space $\L (\X)$ of bounded linear operators on $\X$.

Though Volterra difference systems became an object of active investigations only in last two decades,
they have been appearing in various applications at least since 1930s (see the survey \cite{KC-VT-M03}).
These systems naturally arise in the renewal theory  \cite{F68}, in the numerical studies of
Volterra integral equations \cite{CJRV91}, and in the theory of differential equations
with delays \cite{FMN04_JMJ,FMN04_JDEA} (for a list of other applications see \cite{KC-VT-M03}).
Discretization procedures similar to that of \cite{P80,FMN04_JMJ,FMN04_JDEA} applied to
delayed differential and partial differential equations lead to Volterra difference equations with
infinite-dimensional state spaces $\X$.

Following \cite{M97,CP00,BK10}, we consider equations (\ref{e Keq intro}) and (\ref{e Veq}) in phase space settings.
By $x_n$ the semi-infinite prehistory
sequence $ \left\{ \dots , x (n+m), \dots, x(n-1), x(n) \right\}_{m \le 0}$ is denoted. We suppose that the sequence of initial conditions
$x_0 = \{ x(n+m)\}_{m=-\infty}^{0}$ (i.e., the prehistory of the initial time point $n=0$) belongs to
a certain \emph{phase space} $\B$. In this paper, $\B$ is either one of
exponentially weighted $\ll^p $-spaces $\B^{p,\ga}$
with the norms
$$|x_0|_{\B^{p,\ga}} = \left( \sum_{m=-\infty}^{0} |e^{\ga m} x (m) |_\X^p  \right)^{1/p}  , \ \ \  1 \le p < \infty ,~~~
|x_0|_{\B^{\infty,\ga}} =  \sup_{m \le 0 } |e^{\ga m} x (m) |_\X , \ \ \ p = \infty,
$$
or one of exponentially weighted $c_0 $-spaces $\Bo^{\infty,\ga}$ with the norms of $\B^{\infty,\ga}$
(see the definitions in Section \ref{ss Ph sp aux op}).
Then, all the pre-histories $x_n$ for $n \ge 0$ belong to the same
phase space $\B$.
The system
\begin{equation}
x(n+1) = Q (n) x_n ,  \quad n \geq 0, \label{e h}
\end{equation}
is said to be defined on a phase space $\B$ if $Q (n)$ for all $n \ge 0 $ belong to the space $\L (\B,\X)$
of bounded linear operators from $\B$ to $\X$.
It is clear that system (\ref{e Veq}) can be written in the form (\ref{e h})
if $Q(n,j)$ satisfy certain assumptions that depend on the choice of the phase space $\B$.
For instance, (\ref{e Veq}) is defined on $\B^{p,\ga}$ whenever
$\sum_{j=0}^{+\infty} \|e^{\ga j} Q (n,j) \|_{\X \to \X}^{p'} < \infty$, where $p'$ is the H\"{o}lder's conjugate
of $p$ (i.e., $1/p+1/p'=1$). This is archived by putting $Q (n) x_n = \sum_{j=0}^{+\infty} Q (n,j) x (n-j)$, where
the convergence of the series is understood in the sense of the norm topology of $\X$.
With the same reservations (see also the discussion in Section \ref{ss Ex and Rem conv}),
convolution system (\ref{e Keq intro}) in the phase space settings takes the form of the system
$x (n+1) = K x_n$ with a time invariant coefficient $K \in \L(\B,\X)$.

Usually, in the literature  the phase spaces $\B^{\infty,\ga}$ are used. In \cite{M97}, such spaces are denoted by $\B^\ga$.
When $\ga>0$, these spaces are called (exponentially) fading because of exponentially decaying term $e^{\ga m}$ in the norms.
Following the logic of this terminology, it is natural to say that the phase spaces $\B^{p,\ga}$ with $\ga \leq 0$ are \emph{non-fading}.

In this paper we consider two types of uniform exponential (UE) stability for system (\ref{e h}):
UE stability in $\X$ with respect to (w.r.t.) the phase space $\B$, and UE stability in the sense of resolvent matrix.
The definitions are given in Section \ref{ss stab} in accordance with  \cite{HMN91,CKRV00,FMN04_JMJ,BK10}.
Note that usually exponential stability for Volterra difference systems (\ref{e Veq}) is understood in the following way.

\begin{definition}[see e.g. \cite{EM96}] \label{d UE usual}
System (\ref{e Veq}) is called exponentially stable if there exist constants $C, \nu >0 $ such that,
for any $\tau,s \ge 0 $, the solution $x (n)$ to the problem
\begin{eqnarray*}
x(n+1) & = & \sum_{j=0}^{n+s-\tau} Q (n,j) x (n-j) ,  \quad n \geq \tau, \\
\{ x (\tau - s) , \dots , x (\tau - 1), x (\tau ) \} & = & \{ \vphi^{[-s]}, \dots , \vphi^{[-1]}, \vphi^{[0]} \}
\end{eqnarray*}
with arbitrary initial data $\{ \vphi^{[-j]} \}_{j=0}^{s} \in \X^{s+1}$ satisfies
\begin{equation} \label{e Def es}
 |x(n)|_\X \leq  C e^{-\nu (n-\tau)} \max_{0 \le j \le s} | \vphi^{[-j]} |_\X \ \text{ for all }
n \geq \tau .
\end{equation}
\end{definition}

This type of exponential stability is essentially equivalent to the UE stability in $\X$ w.r.t. $\Bo^{\infty,0}$,
see Remark \ref{r EM=SwrtBo} (1). That is why we will pay special attention to the phase space $\Bo^{\infty,0}$
throughout the paper.

We study stability radii associated with the following perturbation of system (\ref{e Keq intro})
\begin{equation} \label{e K per}
x(n+1)=\sum_{j=0}^{+\infty} \ K(j) \ x(n-j) + D \ N (n) \ E x_n ,
\end{equation}
where $E \in \L (\B,\U_1)$ and  $D \in \L (\U_2 , \X)$. An auxiliary Banach space
$U_2$ ($U_1$) is called the \emph{input  (resp., output) space}.
Perturbed system (\ref{e K per}) can be interpreted as a \emph{feedback system with delayed output}, see Fig.\ref{fig1}.
Note that the output $y (n) = E x_n$ depends on the prehistory $x_n = \{ x (n+m) \}_{m=-\infty}^{0}$ (delayed output)
and that the input $v(n)$ is connected with the output by $ v(n) = N(n) y (n)$, where
$N (n) \in \L (\U_1,\U_2)$ are operators of uncertain feedback (or disturbance operators).

\begin{figure}[h]
\centering

\setlength{\unitlength}{0.0007in}
\begingroup\makeatletter\ifx\SetFigFont\undefined%
\gdef\SetFigFont#1#2#3#4#5{%
  \reset@font\fontsize{#1}{#2pt}%
  \fontfamily{#3}\fontseries{#4}\fontshape{#5}%
  \selectfont}%
\fi\endgroup%
{\renewcommand{\dashlinestretch}{30}
\begin{picture}(7224,2439)(0,-10)
\path(912,2412)(6312,2412)(6312,1212)
    (912,1212)(912,2412)
\path(6312,1812)(7212,1812)(7212,312)(4512,312)
\path(2712,612)(2712,12)(4512,12)
    (4512,612)(2712,612)
\path(2712,312)(12,312)(12,1812)
    (912,1812)(762,1887)
\path(762,1737)(912,1812)
\path(4662,387)(4512,312)(4662,237)
\put(237,2037){\makebox(0,0)[lb]{{\SetFigFont{12}{12}{\rmdefault}{\mddefault}{\updefault}$v(n)$}}}
\put(6537,2037){\makebox(0,0)[lb]{{\SetFigFont{12}{12}{\rmdefault}{\mddefault}{\updefault}$y(n)$}}}
\put(2880,280){\makebox(0,0)[lb]{{\SetFigFont{10}{10}{\rmdefault}{\mddefault}{\updefault} $v(n)=N(n) y (n)$}}}
\put(1362,1750){\makebox(0,0)[lb]{{\SetFigFont{10}{10}{\rmdefault}{\mddefault}
{\updefault}${\displaystyle x(n+1)=\sum_{j=0}^{+\infty} K(j)x(n-j)+Dv(n)}$}}}
\put(1362,1412){\makebox(0,0)[lb]{{\SetFigFont{10}{10}{\rmdefault}{\mddefault}{\updefault}$~~y(n)
= Ex_n$}}}
\end{picture}
}
\caption{Feedback interpretation of system (\ref{e K per})}
\label{fig1}
\end{figure}
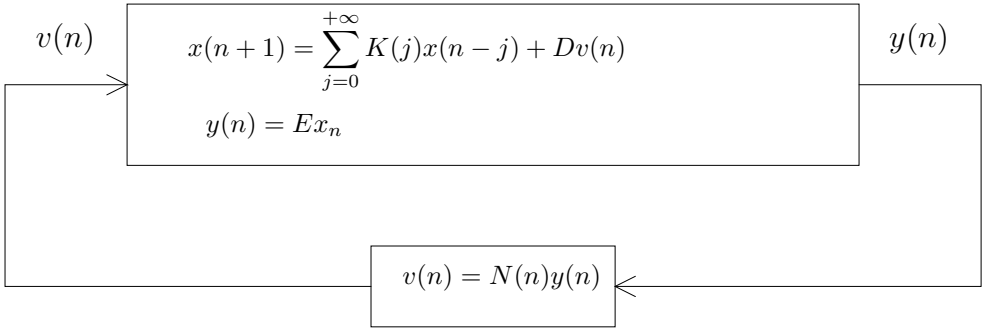

\emph{The (UES) stability radius} $\r_\t$ is, by definition, a sharp bound on the norms of feedback operators $N (n)$ that
ensures UE stability of the perturbed system (\ref{e K per}), see Section \ref{ss r def} for details.
If the feedback operator does not depend on discrete time $n$, $N(n) \equiv \Delta \in \L (\U_1,\U_2)$, one gets \emph{the stability radius w.r.t.
time-invariant structured perturbations} $D \Delta E$. This radius is denoted $\r_\c$. For systems with bounded delay,
more information about structured stability radii and feedback systems
can be found in \cite{PH93,WH94} and references therein.

It seems that, for discrete systems with infinite delay,
the study of stability radii and of very kindred problems of robust stability
was started in the last decade \cite{SB03,SB04,K06,MN08,NNShM09}. This theory is not enough developed yet.
For convolution system (\ref{e Keq intro}), asymptotic stability radii corresponding
to time-invariant structured perturbations were estimated from above in \cite{MN08,NNShM09}.
These papers assume that the coefficients $K(j)$ are either positive operators
on a finite dimensional state space $\X$ \cite{NNShM09}, or are positive compact operators on a complex
 Banach lattice $\X$ \cite{MN08}. Under several additional positivity and compactness assumptions on the perturbations,
 the stability radius is expressed in terms of the Z-transform $\wh K (1)$ of $K (\cdot)$ taken
at the point  $\zeta =1$.

\emph{The main points and results of the present paper are:}

\begin{itemize}
\item Without positivity or compactness assumptions, \emph{two-sided estimates for time-varying exponential stability radii}
$\r_\t$ of convolution system (\ref{e Keq intro})
 are obtained (Theorem \ref{t r ga>0}, Proposition \ref{p r=0},
and Theorem \ref{t r Bo}). We work both with the exponentially fading phase spaces 
and with non-fading phase spaces
$\Bo^{\infty,0}$ and $\B^{\infty,0}$ (presently, the authors do not know any applications of non-fading phase spaces
$\B^{p,\ga}$ with $\ga <0$).
The estimate from above is given in terms of the Z-transform $\wh K (\zeta)$ of the convolution kernel $K (\cdot)$,
the estimate from below via the norm of the input-output operator $ \IO_K $. These results can be seen as an analogue
of the stability radii theory for first order systems, see e.g. \cite{WH94}.

\item For time-invariant radii $\r_\c$, an explicit formula in terms of Z-transform of $K (\cdot)$ is given.
In the case of a Hilbert state space $\X$, we have shown that the same formula is valid for \emph{time-varying} exponential stability radii $\r_\t$
(see formula (\ref{e r=}) and Theorem \ref{t r Bo}).

\item The above mentioned results are used to study unstructured stability radii (Section \ref{s unstr}) and stability radii corresponding to
a feedback scheme with delayed feedback (Section \ref{ss Del feedback}).

\item As a by-product, in formula (\ref{e r=}) and Corollary \ref{c |I-whK|=|IS|},
we establish connections between the norms of transfer functions and the norms of
input-output and unstructured input-state operators.
(The authors believe that such formulae can be obtained in a more straight way. Similar results
are well known for first order system, see e.g. \cite{WH94}.)

\item The results on stability radii are used  to obtain various exponential stability tests for time-varying
Volterra difference systems (\ref{e Veq}), see Section \ref{ss Stab tests}. It seems that at least some of these tests are new
(related problems were actively discussed e.g. in \cite{CKRV00,E09,BK10}).

\end{itemize}

The method used in this paper is a development of that of our previous paper \cite{BK10} and is based on reduction of
system (\ref{e h}) to a first order system $x(n+1)=\A(n)x(n)$. For the study of stability radii w.r.t.
non-fading phase spaces $\Bo^{\infty,0}$ and $\B^{\infty,0}$, we suggest a reduction in two steps:
from systems in non-fading phase spaces to system in exponentially fading phase spaces, and then to first order systems
(see Sections \ref{ss reduction} and \ref{ss p t r Bo}).
To perform this procedure, we fill two following lacunae in the theories of first order and convolution systems:

\begin{itemize}

\item In Section \ref{ss Stab first ord},  we consider first order systems and extend the estimate
$\r_\t \geq \| \IO_\A \|_{\ll^2 (\U_2) \to \ll^2 (\U_1)}^{-1}$
obtained in \cite{WH94} to $\r_\t \geq \| \IO_\A \|_{\ll^q (\U_2) \to \ll^q (\U_1)}^{-1}$ with arbitrary $1\leq q \leq \infty$
(here $\r_\t$ and $\IO_\A$ are the stability radius and the input-output operator corresponding to the first order system, respectively). This extension occurred to be
essential for the study of unstructured stability radii for the convolution system
and stability radii corresponding to delayed feedback.

\item In Section \ref{ss crit conv}, criteria of UE stability of system (\ref{e Keq intro})
are obtained without the assumption of compactness of coefficients $K(j)$.
The usual assumptions of summability of norms $\|K (j) \|_{\X \to \X}$ is also weakened,
for details see the discussion in Section \ref{ss Ex and Rem conv}. One of key tools of the proposed reduction method
is Theorem \ref{t UEScr ga=0}, which shows that the UE stability of convolution system (\ref{e Keq intro}) w.r.t. the non-fading phase spaces
$\Bo^{\infty,0}$ and  $\B^{\infty,0}$ is equivalent to that w.r.t. fading phase spaces $\B^{p,\ga}$
with small positive $\ga$. Continuing the program of \cite{BK10}, we also obtain an
 exponential stability criterion of Bohl-Perron type for system (\ref{e Keq intro}) in $\B^{\infty,0}$,
 see Corollary \ref{c BPcr ga=0}.
It seems that, for time-varying systems in $\B^{\infty,0}$, finding of similar criteria is still an open problem.
\end{itemize}

Another key point of the present paper is the use of phase spaces $\B^{2,\ga}$. The spaces $\B^{2,\ga}$
are Hilbert spaces whenever the state space $\X$ is Hilbert. This fact, in combination with embedding (\ref{e emB})
of phase spaces, allows us to give an explicit expression for time-varying radii $\r_\t$ in formula (\ref{e r=}), Theorem
\ref{t r Bo}, and Corollaries \ref{c unstr}, \ref{c |I-whK|=|IS|}, \ref{c del feedback ga>0}.

The paper is organized as follows. After introducing notations and basic stability definitions in Section \ref{s Prelim},
we present stability results concerning convolution system (\ref{e Keq intro}) in Section \ref{s conv}.
Section \ref{ss Ex and Rem conv} provides examples to these results and discusses connections with previous
studies \cite{EM96,FMN04_JMJ,FMN04_JDEA,MN05}. In Section \ref{s pert},  perturbed systems are considered:
after introducing perturbation types, stability radii, and input-output operators in Section \ref{ss r def},
the stability radii are estimated in Section \ref{ss radii est}.
The proofs of two main results of Section \ref{s pert}, Theorems \ref{t r ga>0} and
\ref{t r Bo}, are given in Section \ref{s proof radii}, which constitutes the main technical part of the paper.
Section \ref{s unstr} deals with the important special case of unstructured perturbations.
Section \ref{s appl} presents some applications and examples which illustrate the obtained criteria
and estimates: various stability tests for time-varying Volterra difference systems are derived in
Section \ref{ss Stab tests}, stability radii associated with delayed feedback are considered in Section \ref{ss Del feedback},
and, finally, Section \ref{ss ex nonpos} provides an example of calculation of stability radii for a non-positive system.

\section{Notation and basic definitions \label{s Prelim}}

We use the convention that the sum equals zero if the lower index exceeds
the upper index.

For a set $S$ in a normed space, $\overline{S} $ is its closure.
The following sets of complex numbers are used: $\TT:=\{\zeta \in \CC : |\zeta| =1 \}$,
$\DD (\varrho):=\{ \zeta \in \CC : |\zeta|<\varrho \}$, and $\overline{ \DD (\varrho) } := \{ \zeta \in \CC : |\zeta| \leq \varrho \}$.
By $\ZZ$ and $\ZZ^+$ ($\ZZ^-$), the sets of all integers and all nonnegative
(resp., nonpositive) integers are denoted. We write $\ZZ^+_\tau$ for the infinite interval of
integer numbers in $[\tau, +\infty)$. So $\NN = \ZZ_1^+$.

\subsection{Phase spaces, Z-transform, auxiliary operators and functions}
\label{ss Ph sp aux op}

Let $\U$, $\U_1$, $\U_2$ be Banach spaces.
The norm in $\U$ is denoted by $| \cdot |_{\U}$.
Then $\Ssp ( \U )$ ($\Ssp_\pm ( \U )$)
denotes the vector space of all
\emph{discrete functions} $v:\ZZ \to \U$ (resp., $v:\ZZ^\pm \to \U$). Further,
$\L (\U_1 , \U_2 )$ denote the Banach
space of bounded linear operators from $\U_1 $ to $\U_2$, $\| \cdot \|_{\U_1 \to \U_2}$ is the corresponding norm.

The zero vector of a vector space $\W$ is denoted by $0_\W$,
the identity (zero) operator in $\W$ by $I_\W$ (resp., $\Oz_\W$).

An operator $G \in \L(\U) :=\L (\U,\U)$ is called \emph{boundedly invertible} if it is invertible and $G^{-1} \in  \L(\U)$.
The \emph{kernel} of an operator $G \in \L(\U)$ is denoted by
$\ker G := \{ u \in \U : Gu = 0_\U\}$, and the \emph{image} of $G$ is
\[
\ran G := \{ u \in \U : u = Gv \text{ for certian } v \in \U \}.
\]
The duals of the space $\U$ and of the operator $G$ are denoted by
$\U^*$ and $G^*$.

We will use the standard  Banach spaces $ \ll^p ( \U ) = \ll^p (\ZZ^+, \U ) $ of discrete $\U$-valued
$\ll^p$-functions (so $ \ll^p ( \U ) \subset \Ssp_+  ( \U ) $).


The (unilateral) \emph{Z-transform} of a discrete function $ u : \ZZ^+ \to \U$
is understood as the power series
\begin{equation} \label{e Ztr}
\wh u (\zeta) = \sum_{j=0}^{+\infty} \zeta^j u(j)
\end{equation}
and, simultaneously, as the corresponding $\U$-valued function defined on its set of convergence in $\CC$, see e.g. \cite{K81}.
This definition is common for Geophysics. In some papers, the Z-transform of $u$ is defined as $\wh u  (\zeta^{-1})$, see e.g. \cite{EM96}.
These definitions are equivalent, but the first is more convenient for us since,
if \emph{the corresponding convergence radius}
\begin{equation} \label{e R}
R [\wh u] := \left( \limsup_{j\to +\infty} | u (j) |_{\U}^{1/j} \right)^{-1}
\end{equation}
is positive, then $\wh u$ is analytic in the open disc $\DD (R[\wh
u]) $.

A function $u : \ZZ^+ \to \U $ is said to \emph{decay exponentially} if
$ | u (j) |_\U \leq C e^{-\gamma j}$ for some $\gamma , C > 0$ (this is equivalent to $R [\wh u]>1$).

Let a (nontrivial) Banach space $\X$ be our state space.
By $\X^{\ZZ^-}$ we denote the vector space of semi-infinite tuples $\vphi = \{ \vphi^{[m]} \}_{m=-\infty}^{0}$
with elements $\vphi^{[m]}$ in $\X$ and indices $m$ in $\ZZ^-$.
We will say that $\vphi^{[m]}$ is \emph{the m-th coordinate} of
$\vphi$. The standard notation where $\Ssp_- (\X)$ is used instead of $\X^{\ZZ^-}$, see e.g. \cite{M97}, is
inconvenient in the context of the reduction method used in the present paper.

For  $m \in \ZZ^- $ we define \emph{the coordinate operator}
$ P_m : \X^{\ZZ^-}  \to \X$ by
$P_m \vphi = \vphi^{[m]} $.
The operator-valued matrix corresponding to $P_m$ is the row matrix $ ( \dots, \Oz_\X, I_\X, \Oz_\X, \dots,  \Oz_\X)$
with the only non-zero entry at $m$-th position.
The transpose column matrix defines the operator $ P_m^\T :\X \to \X^{\ZZ^-}$, i.e.,
\begin{equation} \label{e PmT}
 ( P_m^\T \psi)^{[j]}  = \left\{ \begin{array}{rr}
\psi , & j = m \\
0_\X  , & j \neq m
\end{array} \right. , \qquad \psi \in \X .
\end{equation}

A linear subspace $\B$ of $\X^{\ZZ^-}$ satisfying a certain set of axioms
is called a phase space (see e.g. \cite{M97}).
We will not discuss those axioms since we consider only exponentially weighted $\ll^p$- and $\c_0$-type phase spaces:
\begin{eqnarray*}
\B^{p,\ga} & := & \left\{ \ \{ \vphi^{[m]} \}_{m=-\infty}^{0} \in \X^{\ZZ^-} \ :
\ | \vphi |_{\B^{p,\ga}} :=
\left( \sum_{m=-\infty}^{0} |e^{\ga m} \vphi^{[m]} |_\X^p  \right)^{1/p} \!\! < \infty
\right\},
\quad 1 \leq p < \infty , \\
\B^{\infty,\ga} & := & \left\{ \ \{ \vphi^{[m]} \}_{m=-\infty}^{0} \in \X^{\ZZ^-} \ :
\ | \vphi |_{\B^{\infty,\ga}} := \sup_{m \in \ZZ^-} |e^{\ga m} \vphi^{[m]} |_\X  < \infty \right\} , \\
\Bo^{\infty,\ga} & := & \{ \vphi \in \B^{\infty,\ga} \ : \ \lim_{m\to -\infty} | e^{\ga m} \vphi^{[m]} |_\X = 0  \} ,
\quad \ \quad | \cdot |_{\Bo^{\infty,\ga}} := | \cdot |_{\B^{\infty,\ga}} .
\end{eqnarray*}
In the notation of \cite{M97}, $\B^{\infty,\ga}$ is $\B^{\ga}$.
The spaces $\B^{p,\ga}$ with $p \in [1,\infty)$ were considered in \cite{BK10}.
We will systematically use the fact that
if $\X$ is a Hilbert space, then so are $\B^{2,\ga}$.

The considered class of phase spaces is totally ordered by the continuous embedding:
\begin{equation} \label{e emB}
 \B^{\infty,\ga_{\scriptstyle 0}} \subset \B^{p_{\scriptstyle 0},\ga} \subset \B^{p_{\scriptstyle 1},\ga} \subset
 \Bo^{\infty,\ga} \subset \B^{\infty,\ga} \subset \B^{1,\ga_{\scriptstyle 1}} ,
 \quad  1 \le p_0 < p_1 < \infty, \ \ga_0 < \ga < \ga_1.
\end{equation}

For a function $x (\cdot) \in \Ssp (\X)$, $x_n \in \X^{\ZZ^-}$ denotes the prehistory of $x(n)$, i.e.,
$ x_n^{[m]} = x (n+m) $, $ m \in \ZZ^- $. One can see that $x_n \in \B$ yields $x_{n+1} \in \B$ for any phase space $\B$.

For an operator $G \in \L (\B, \U)$, let us define a discrete function
$G (\cdot) \in \Ssp_+ \left( \L (\X,\U)
\right)$ by
\begin{equation} \label{e def G dot}
 G (n) = G P_{-n}^\T , \ \ n \in \ZZ^+, \text{ and the associated Z-transform  by } \
\wh G (\zeta) = \sum_{j=0}^{+\infty} \zeta^j G (j) .
\end{equation}

\begin{remark} \label{r p=inf repr}
The definition of $G (\cdot)$ is justified by the following: if $\B = \B^{p,\ga}$ with $p \in [1,\infty)$ or
$\B = \Bo^{\infty,\ga}$,
then
$ G \vphi = \sum_{j=0}^{+\infty} G (j) \vphi^{[-j]} $
for all $ \vphi \in \B$, where the infinite sum is understood in the sense of the strong topology of $\U$.
When $p=\infty$, this representation of $G$ does not hold for certain $G \in \L (\B^{\infty,\ga}, \U)$ and
$\vphi \in \B^{\infty,\ga}$. Such $G$ and $\vphi$ can be constructed, e.g., using Banach
limits, see \cite[Remark 2.9]{BK10} (and also Example \ref{ex l inf} below for another related effect).
\end{remark}

\begin{lemma} \label{l Kn est}
Let $\B=\Bo^{\infty,\ga}$ or  $\B = \B^{p,\ga}$ with $1 \le p \le \infty$.
Assume $G \in \L ( \B, \U)$. Then \linebreak $ \| G (n) \|_{\X \to \U} \leq e^{-\ga n} \| G \|_{\B \to \U}$
for all $n \in \ZZ^+$
and $\ga$ such that $e^\ga \leq R [\wh G ] $.
\end{lemma}

\begin{proof}
The $\B$-norm of the tuple $\vphi = \{ \vphi^{[m]} \}_{m=-\infty}^{0} = \{ \dots, 0_\X, 0_\X, \psi, 0_\X, 0_\X, \dots  \}$
with the only nonzero entry at $m_0$-th position is $| \vphi |_\B = e^{\ga m_0} | \psi |_\X $ (note that $m_0 \leq 0$).
Since $ G (-m_0) \psi = G P_{m_0}^\T \psi = G \vphi $, we see that
$\| G (-m_0) \|_{\X \to \U} \leq e^{\ga m_0} \| G \|_{\B \to \U}$ for all $m_0 \in \ZZ^-$.
Plugging this into (\ref{e R}), one gets $R [\wh G ] \geq e^\ga $.
\end{proof}

\subsection{Stabilities and the input-state operator}
\label{ss stab}

We say that $Q(\cdot)$ \emph{defines the system (\ref{e h}) on a phase space $\B$} if $Q (n) \in \L (\B,\X)$ for all
$n \in \ZZ^+$.

From now on assume that  $Q (\cdot)$ defines system (\ref{e h}) on a certain phase space $\B$.
Then, for any $(\tau,\vphi) \in \ZZ^+ \times \B$, there exists  unique $x : \ZZ \to 
\X$
such that $x_\tau = \vphi $ and (\ref{e h}) holds for all $n \geq \tau$.
The function $x$ is called a solution to (\ref{e h}) through
$(\tau,\vphi)$, and is denoted by $x (\cdot, \tau, \vphi)$.
For each $n \in \ZZ$, $x_n (\tau, \vphi) := \{ x (n+m, \tau, \vphi) \}_{m=-\infty}^{0}  \in \B $.

Define \emph{the resolvent (fundamental) matrix}
$\{ X_Q (n,\tau) \}_{n \geq \tau \geq 0}$ by the equalities
\begin{equation}
 X_Q(n,\tau) \psi := x (n,\tau,P_0^\T \psi), \ \ \ \psi \in \X,
\end{equation}
recall that $x (\cdot ,\tau,P_0^\T \psi)$ is the solution to  (\ref{e h}) satisfying
\[
\{ \dots, x (\tau-2), x (\tau-1), x (\tau) \} = \{ \dots, 0_\X, 0_\X, \psi \}.
\]
So $ X_Q (n,\tau) \in \L (\X)$.

Let us define \emph{an unstructured input-state operator} $\IS_Q : \Ssp_+ (\X) \to
\Ssp_+ (\X)$ by $\IS_Q (f (\cdot))= x (\cdot)$, where $x=x(\cdot)$
is the solution to the nonhomogeneous system
\begin{equation} \label{e Qnj eq f}
x_0 =0_\B  , \ \ \  x(n+1)=\sum_{j=0}^{+\infty}  Q(n) x_n + f(n), \quad n\geq 0 .
\end{equation}
The unstructured input-state operator and the resolvent matrix are connected by
\begin{equation} \label{e IS=sumQ}
(\IS_Q f) (n) = \sum_{j=0}^{n-1} X_Q (n,j+1) f (j)  \qquad \text{ (see e.g. \cite{CKRV00})} .
\end{equation}

\begin{definition} \label{d UESn}
System (\ref{e h}) is called
\emph{uniformly exponentially stable (UES, in short) in (the sense of) $\X$
with respect to a phase space $\B$}  if it is defined on $\B$ and there exist constants $C, \nu >0 $  such that
\begin{equation}
|x(n, \tau, \vphi)|_\X \leq  C e^{-\nu (n-\tau)} | \vphi |_\B \quad
\text{for all} \ n, \tau \ \text{such that} \ n \geq \tau \geq 0 \ \text{ and } \vphi \in \B.
\label{e ExStX}
\end{equation}
\end{definition}

This stability definition for Volterra systems modifies
that of the first order case following the lines of \cite{HMN91}
and \cite{BK10}.

In \cite{CKRV00}, the exponential stability is understood in the
resolvent matrix sense.

\begin{definition}[\cite{CKRV00}] \label{d UES XX}
System (\ref{e h}) is called \emph{UES in the
resolvent matrix sense}
if there exist $C, \nu >0$ such that $|x(n, \tau, P_0^\T
\psi)|_\X \leq  C e^{-\nu (n-\tau)} | \psi |_\X $ for all $n \geq \tau
\geq 0$ and $\psi \in \X$, or equivalently,
\begin{equation}
\| X_Q (n,\tau) \|_{\X \to \X} \leq  C e^{-\nu (n-\tau)} . \label{e
ExSt XX}
\end{equation}
\end{definition}

\begin{remark} \label{r EM=SwrtBo}
(1) The exponential stability introduced by Definition \ref{d UE usual} is equivalent to the
UE stability in $\X$ w.r.t. $\Bo^{\infty,0}$ in the following sense.
One can define operators $Q(n)$ on finite
tuples $\vphi = \{ \vphi^{[m]} \}_{m=-\infty}^{0}$ (i.e., on tuples that have a finite number of nonzero entries)  by  $ Q (n,j) \vphi = \sum Q (n,j) \vphi^{[-j]} $.
Assume that system (\ref{e Veq}) is exponentially stable. Then, operators $Q(n)$ have 
a dense in $\Bo^{\infty,0}$ domain
and, by (\ref{e Def es}), are bounded as operators from $\Bo^{\infty,0}$ to $\X$. So they can be extended by continuity to the whole space
$\Bo^{\infty,0}$. The resulting system (\ref{e h}) is UES in
$\X$ w.r.t. $\Bo^{\infty,0}$. Indeed, (\ref{e Def es}) implies (\ref{e ExStX}) for finite tuples $\vphi$,
and passing to limit one can extend (\ref{e ExStX}) to all
tuples $\vphi \in \Bo^{\infty,0}$. Inverting this procedure, one can immediately see that each
system (\ref{e h}) that is UES in $\X$ w.r.t. $\Bo^{\infty,0}$ produces an exponentially stable system (\ref{e Veq}).

(2) Clearly, for every phase space $\B$, the UE stability in $\X$ w.r.t. $\B$ implies the UE stability  in the
resolvent matrix sense.
\end{remark}

For $\B = \B^{p,\ga}$, the following criterion of Bohl-Perron type is a reformulation of \cite[Theorems
3.1 and 7.2]{BK10} (see also \cite{BB06} for $q_1=q_2=\infty$). Clearly, the proof of \cite[Theorem 3.1]{BK10}
works for $\B=\Bo^{\infty,\ga}$ as well.

\begin{theorem}[\cite{BK10}] \label{t BP ga>0}
Let $\ga>0$. Let $\B=\Bo^{\infty,\ga}$ or $\B = \B^{p,\ga}$ with $1 \le p \le \infty$.
Let the (ordered) pair $(q_1,q_2) \neq (1,\infty)$ be such that $1 \leq q_1 \leq q_2 \leq
\infty$. Then
\begin{equation} \label{e BPcr}
\text{(\ref{e h}) is UES in } \X \text{ w.r.t. } \B \ \
\Leftrightarrow  \ \ \IS_Q  \in \L (\ll^{q_1} (\X), \ll^{q_2} (\X)) \text{ and }
\sup_{n \in \ZZ^+} \| Q (n) \|_{\B \to \X} < \infty.
\end{equation}
\end{theorem}

The proofs of \cite[Theorems 3.1 and 7.2]{BK10}
essentially use the assumption that the phase space is exponentially fading
(i.e., $\ga >0$). Among other results of the next section, we give a Bohl-Perron type criterion for
Volterra systems of convolution type in the non-fading phase space $\Bo^{\infty,0}$.

\begin{remark}
Other types of connections between stability and properties of unstructured input input-state operator were
considered in \cite{CKRV00,SB03,SB04}.
\end{remark}

\section{UE-stability for Volterra systems of convolution type}
\label{s conv}

\subsection{Criteria of UE-stability in $\Bo^{\infty,0}$ and in fading phase spaces}
\label{ss crit conv}

Let $\ga \in \RR$. Let $\B=\Bo^{\infty,\ga}$ or  $\B = \B^{p,\ga}$ with $1 \le p \le \infty$.
Assume $K \in \L ( \B, \X)$ and
let $K (\cdot) $  be the
 associated discrete function  defined by  (\ref{e def G dot}).
In the case $\B=\B^{\infty,\ga}$, we impose the additional technical assumption that
\begin{multline} \label{e as K=sum}
K \vphi =  \sum_{j=0}^{+\infty} \ K(j) \ \vphi^{[-j]} \ \ \text{ for all } \vphi \in \B^{\infty,\ga},
\text{ where the infinite sum is understood in the sense}
\\ \text{  of the norm topology of } \X.
\end{multline}
Note that for the other phase spaces $B$, this assumption is always fulfilled due to Remark \ref{r p=inf repr}.

Recall that $\wh K (\zeta) $ is the Z-transform of the discrete function $K(\cdot)$.
Lemma \ref{l Kn est} implies $R [\wh K ] \geq e^\ga $.
That is, the Z-transform $ \wh K (\cdot) $ is analytic in $\DD (e^\ga)$.
The sum $\sum_{j=0}^{+\infty} \ K(j) \ \vphi^{[-j]} $ defines a continuous operator on
$\B^{p,\ga_1}$ with any $\ga_1 < \ln R [\wh K ]$ (in particular, with any $\ga_1 < \ga$).
We keep the same notation $K$ for all these operators.

In this section, we study the Volterra system of convolution type
\begin{equation} \label{e Keq}
 x(n+1)=\sum_{j=0}^{+\infty} \ K(j) \ x(n-j) , \quad n\geq 0 .
\end{equation}
In our settings, this system can be written in the form $x (n+1) = K x_n$ and is defined on the phase space $\B$ (as well as on the phase spaces
$\B^{p,\ga_1}$ with $\ga_1 < R [\wh K ]$).

Recall that the unstructured input-state
operator $\IS_K $ associated with (\ref{e Keq}) is defined by $\IS_K (f (\cdot))= x (\cdot)$, where $x=x(\cdot)$
is the solution to the nonhomogeneous system
\begin{equation} \label{e Keq f}
x(n+1)=\sum_{j=0}^{n} \ K(n-j) \ x(j) + f(n) , \quad n\geq 0, \ \ \
x (0) =0_\X .
\end{equation}
Recall also that an operator $G \in \L(\X) $ is called boundedly invertible if
$\ker G = \{ 0_\X \}$ and $G^{-1} \in  \L(\X)$.

\begin{theorem} \label{t UEScr ga>0}
Let $\ga >0$. Let system (\ref{e Keq}) be defined on $\B$, where $\B=\Bo^{\infty,\ga}$ or $\B = \B^{p,\ga}$ with $1 \le p \le \infty$. Let
\begin{equation} \label{e q1q2neq}
1 \le q_1 \le q_2 \le \infty \qquad \text{and} \qquad  (q_1,q_2)
\neq (1,\infty) .
\end{equation}
Then, the following statements are equivalent:
\item[(i)]  System (\ref{e Keq}) is UES in $\X$ w.r.t. $\B$.
\item[(ii)] For all $\zeta  \in \overline{\DD (1)} $, the operators $ I_\X - \zeta \wh K (\zeta) $ are boundedly invertible.
\item[(iii)] System (\ref{e Keq}) is UES in the resolvent matrix sense.
\item[(iv)] $\IS_K \in \L \left( \ll^{q_1}, \ll^{q_2} \right) $.
\end{theorem}

The proof is given in Section \ref{ss proof of UEScr}. A connection of
$ \max_{|\zeta|=1} \| [I_\X - \zeta  \wh K (\zeta)]^{-1} \|_{\X \to \X} $
and $ \| \IS_K \|_{ \ll^q (\X) \to \ll^q ( \X) }  $ is considered in Corollary \ref{c |I-whK|=|IS|}.
Under certain additional assumptions, the equivalencies (i) $\Leftrightarrow$ (ii) $\Leftrightarrow$ (iii)
were obtained in \cite[Theorems 1 and 2]{FMN04_JMJ}, see  for details Remark \ref{r comp ga>0} below.

\begin{proposition} \label{p st emb}
Assume that for another phase space $\B_1$ the
continuous embedding $\B_1 \subset \B$ holds. Then if system (\ref{e h}) is UES
in $\X$ w.r.t. $\B$, it is UES in $\X$ w.r.t. $\B_1$.
\end{proposition}

For the proof, note that the embedding implies that the system is defined on $\B_1$. Now
the statement follows immediately from the UE stability definition and the continuous embedding inequality
$| \cdot |_{\B} \leq C | \cdot |_{ \B_{\scriptstyle 1}}$.

The main result of this section is that, for system (\ref{e Keq}) defined on $\Bo^{\infty,0}$,
\emph{this proposition can be partially reversed}.

\begin{theorem} \label{t UEScr ga=0}
Let (\ref{e Keq}) be
defined on $\B=\Bo^{\infty,0}$ or on $\B = \B^{\infty,0}$.
Then, the following statements are equivalent:
\item[(i)]  System (\ref{e Keq}) is UES in $\X$ w.r.t.  $\B$.
\item[(ii)] There exists $\ga_0 >0$ such that (\ref{e Keq}) is UES in $\X$ w.r.t. $\B^{p,\ga}$
for all $(p,\ga) \in [1,\infty] \times (0,\ga_0]$.
\item[(iii)] System (\ref{e Keq}) is UES in the resolvent matrix sense.
\end{theorem}

The proof is given in Section \ref{ss proof of UEScr}.

Note the following simple fact:
\begin{multline} \label{e RwhK>1 equiv}
 K (\cdot) \text{ decays exponentially } \ \Longleftrightarrow \ \text{ there exists } \ga_0>0 \text{ such that } K \in \L (\B^{p,\ga},\X) \\
\text{ for all }
(p,\ga) \in [1,\infty] \times (0,\ga_0]
\end{multline}
(obviously, 'for all $(p,\ga) \in $ ...' can be replaced by 'for a certain pair $(p,\ga) \in $ ...' saving the equivalence).
This fact together with Theorems \ref{t UEScr ga=0} and \ref{t UEScr ga>0}
implies immediately the following statement, which may be considered as
\emph{a Bohl-Perron type criterion for  Volterra system of convolution type in $\Bo^{\infty,0}$}.

\begin{corollary} \label{c BPcr ga=0}
Let $q_1$ and $q_2$ satisfy (\ref{e q1q2neq}). Then, system (\ref{e Keq})
is UES in $\X$ w.r.t. $\Bo^{\infty,0}$ (w.r.t. $\B^{\infty,0}$) if and only if
$\IS_K \in \L \left( \ll^{q_1}, \ll^{q_2} \right) $ and $K (\cdot)$ decays exponentially.
\end{corollary}

\begin{corollary} \label{c Zcr ga=0}
System (\ref{e Keq}) is UES in $\X$ w.r.t. $\Bo^{\infty,0}$ ( w.r.t. $\B^{\infty,0}$) if and only if
$K (\cdot)$ decays exponentially and the operators $ I_\X - \zeta \wh K (\zeta) $ are boundedly invertible for all $\zeta  \in \overline{\DD(1)} $ .
\end{corollary}

In the case $\X=\CC^n$, Corollary \ref{c Zcr ga=0} was obtained in \cite[Theorems 5 and 2]{EM96}.
Note that when $\X$ is finite-dimensional,
the condition that $ I_\X - \zeta \wh K (\zeta) $ is boundedly invertible for $\zeta  \in \overline{\DD(1)} $ turns into
the condition
\[
\det [I_\X - \zeta \wh K (\zeta)] \neq 0 ,   \ \ \ \ \zeta  \in \overline{\DD(1)}
\]
of \cite[Theorem 2]{EM96}.
In the case when $\X$ is a Banach space and the operators $K(j)$ are compact,
a statement close to Corollary \ref{c Zcr ga=0} follows from \cite[Theorems 4 and 2]{FMN04_JDEA},
see for details Remark \ref{r comp ga=0} below.

\subsection{Proofs of Theorems \ref{t UEScr ga>0} and \ref{t UEScr ga=0}}
\label{ss proof of UEScr}


Let $\ShEl$ be the right shift in $\Ssp_+ (\X)$, i.e., $(\ShEl x)
(j) = \left\{ \begin{array}{ll} 0, & j=0 \\
x (j-1) , & j \in \NN
\end{array} \right. $.
By $\ShEl^\T$ we define the operator with the transpose $\L(\X)$-valued matrix, i.e.,
\[
(\ShEl^\T x) (j) = x (j+1) \ \ \ \text{ for all } j \in \ZZ^+ .
\]
In other words, $\ShEl^\T$ is the backward shift with truncation of the coordinate with the negative index $-1$.

Obviously, for arbitrary $K (\cdot) \in \Ssp_+ (\L(X))$,
\begin{equation} \label{e IS bijec}
\ShEl^\T \IS_K \ \text{ is a self-bijection of } \ \Ssp_+ (\X)
\end{equation}
(here $\IS_K$ is the unstructured input-state operator defined via  (\ref{e Keq f})).

For convolution system (\ref{e Keq}) the resolvent matrix is
a Toeplitz matrix, i.e., $ X_K (n,j) = X_K (n-j)$ with $X_K
(\cdot)\in \Ssp_+ (\L (\X))$.
In particular, one can define the
Z-transform $\wh X_K (\zeta)$ of $X_K (\cdot)$ (at least as a formal
power series).

In the next lemma, assertions (ii) and (iii) are understood in
the power series sense,  $(\ShEl^\T x)\wh \ (\zeta) $ is the Z-transform of  $(\ShEl^\T x) (\cdot)$
defined by (\ref{e Ztr}).

\begin{lemma} \label{l IS equiv}
For systems of convolution type, the following assertions are equivalent:
\item[(i)] $x (n) = (\IS_K f) (n)$, $n \ge 0$,
\item[(ii)] $ [I_\X - \zeta \wh K (\zeta)] \ (\ShEl^\T x)\wh \ \, (\zeta) = \wh
f(\zeta)$ and $x(0) = 0_\X$,
\item[(iii)] $ (\ShEl^\T x)\wh \ \, (\zeta)   = \wh X_K (\zeta) \ \wh f (\zeta)$
 and $x(0) = 0_\X$.
\end{lemma}

\begin{proof}
\textbf{(i) $\Leftrightarrow$ (ii)}. System (\ref{e Keq f})
implies $(\ShEl^\T x) (n)= (K * x) (n) + f (n) $, where '*' stands for convolution. Applying the Z-transform
and taking into account the fact that $\wh x(\zeta) = \zeta (\ShEl^\T x)\wh \ \, (\zeta)$ for $x$ such that $x (0) =
0_\X$, we get $(\ShEl^\T x)\wh \ \, (\zeta) =
\zeta \wh K (\zeta) \ (\ShEl^\T x)\wh \ \, (\zeta) + \wh f (\zeta) $.   This yields (ii).
Inverting the above calculations, we see that (ii) $\Rightarrow$ (i).

The equivalence \textbf{(i) $\Leftrightarrow$ (iii)} follows from (\ref{e IS=sumQ}), which, for system (\ref{e Keq f}),
 takes the form
\begin{equation} \label{e IS K conv}
(\IS_K f) (n) = \sum_{j=0}^{n-1} X_K (n-j-1) f (j) .
\end{equation}
\end{proof}

Put 
\[
R_{\min} := \min \{ R (\wh K), R (\wh X_K) \} .
\]

\begin{lemma}  \label{l whK=whXK new}
$(i)$ $I_\X - \zeta \wh K (\zeta)$ and $\wh X_K (\zeta)$ are two-sided inverses
to each other in the ring of formal power series, i.e.,
\begin{equation} \label{e K XK inv}
 [I_\X - \zeta \wh K (\zeta)] \wh X_K (\zeta) \equiv \wh X_K (\zeta) [I_\X - \zeta \wh K (\zeta)] \equiv I_\X .
\end{equation}

$(ii)$ For all $\zeta \in \DD (R_{\min})$, the operator $I_\X - \zeta \wh K (\zeta) $ is boundedly invertible and
$
[I_\X - \zeta \wh K (\zeta)]^{-1} = \wh X_K (\zeta) .
$
\end{lemma}

\begin{proof}
It is enough to prove (i), statement (ii) follows immediately from (i).
By (\ref{e IS bijec}), we can test statements (ii) and (iii) of
Lemma \ref{l IS equiv} with arbitrary $x \in \Ssp_+ (\X)$
satisfying $x(0)=0_\X$ or with arbitrary $f \in \Ssp_+ (\X)$.

Testing with $\{ x (0), x(1), x(2) ,  x(3) \dots \} = \{ 0_\X, \psi , 0_\X , 0_\X , \dots \}$,
where $\psi \in \X$ is arbitrary,
we see that:
\begin{itemize}
\item the power series  $(\ShEl^\T x)\wh \ (\zeta)$ has only the zero-order term $ \zeta^0 \psi\equiv \psi$,
\item $ [I_\X - \zeta \wh K (\zeta)] \ \psi \equiv \wh f(\zeta)$,
\item and $ \psi \equiv \wh X_K (\zeta) \ \wh f (\zeta)$.
\end{itemize}
 Combining the two last equalities, one gets $ \psi \equiv \wh X_K (\zeta) \ [I_\X - \zeta \wh K (\zeta)] \ \psi$.
This implies $ X_K (\zeta) [I_\X - \zeta \wh K (\zeta)] \equiv I_\X $.

Testing (ii) and (iii) of Lemma \ref{l IS equiv} with $\{ f (0), f(1), f(2) , \dots \} = \{ \psi , 0_\X , 0_\X , \dots \}$, we see that
$  [I_\X - \zeta \wh K (\zeta)] \wh X_K (\zeta) \equiv I_\X $.
\end{proof}

Note that (\ref{e ExSt XX}) implies the following equivalence
\begin{equation} \label{e R>1=UESXinX}
(\ref{e Keq}) \text{ is UES in the resolvent matrix sense} \ \ \Leftrightarrow \ \ R [\wh X_K] > 1.
\end{equation}

\begin{proof}[Proof of Theorem \ref{t UEScr ga>0}]
First recall that $R [\wh K] \geq e^\ga > 1$ since (\ref{e Keq}) is defined on $\B$ with $\ga>0$.

\textbf{(i) $\Rightarrow$ (iii)}. Plugging $ \vphi = \{ \dots, \vphi^{[-2]}, \vphi^{[-1]}, \vphi^{[0]} \} = \{ \dots, 0_\X, 0_\X, \psi \} $ into (\ref{e ExStX}), we see
that the condition of Definition \ref{d UES XX} is satisfied.

\textbf{(iii) $\Rightarrow$ (ii)}. By (\ref{e R>1=UESXinX}), $R [\wh X_K] >1$.  Thus,
$R_{\min} >1$ and Lemma \ref{l whK=whXK new} (ii) completes the proof.

\textbf{(ii) $\Rightarrow$ (iii)}. By \cite[Sec.
VII.6]{DSh58}, the set of $\zeta \in \DD (R [\wh K])$ such that $I_\X -
\zeta \wh K (\zeta)$ is boundedly invertible is open, and, moreover,
$[I_\X - \zeta \wh K (\zeta)]^{-1}$ is analytic on this set. Since
$R [\wh K]>1$, (ii) implies that $[I_\X - \zeta \wh K
(\zeta)]^{-1} $ is analytic in $\DD (\varrho)$ with certain $\varrho > 1$.
This and (\ref{e K XK inv}) imply $R [\wh X_K] \ge \varrho >1$, and, due to (\ref{e R>1=UESXinX}), statement (iii).

\textbf{(iii) $\Rightarrow$ (iv)}. It is enough to consider the case $q_1 = q_2=q$.
By (iii), $ X_K (n-j) = X_K (n,j) $ satisfy  (\ref{e ExSt XX}). In particular,
$X_K (\cdot) \in \ll^1 \left( \L (\X) \right)$.
Applying Young's inequality for convolutions (see e.g. \cite[Problem
VI.11.10]{DSh58}) to (\ref{e IS K conv}), one obtains
$ | (\IS_K  f) (\cdot) |_{\textstyle \ll^{q}} \leq  | X_K (\cdot)
|_{\textstyle \ll^1} | f (\cdot) |_{\textstyle \ll^q} $.

\textbf{(iv) $\Rightarrow$ (i)} due to Theorem \ref{t BP ga>0}.
\end{proof}

\begin{lemma} \label{l G bound inv}
Assume that $\varrho >1$ and that an $\L(\X)$-valued function $G$ is analytic in $\DD (\varrho)$, boundedly invertible in
$\DD (1)$, and $\sup_{\zeta \in \DD (1)} \| G^{-1} (\zeta) \|_{\X \to \X} = C < \infty$. 
Then, $G$ is boundedly invertible
in a certain open neighborhood of $\overline{\DD (1)}$.
\end{lemma}

\begin{proof}
It follows from the assumptions, that $C>0$ and the following inequalities hold
\begin{equation} \label{e G>||}
| G(\zeta) \psi |_\X \geq C^{-1} |\psi |_\X  , \ \ 
\left| [G(\zeta)]^* \psi^* \right|_{\X^*} \geq C^{-1} |\psi^* |_{\X^*} \
\text{ for all } \zeta \in \DD (1), \psi \in \X, \psi^* \in \X^* .
\end{equation}
For arbitrary $\zeta_0 \in \TT \ = \{ |z| =1 \}$, let us take $\{ \zeta_n \} \subset \DD (1)$
such that $\zeta_n \to \zeta_0$ as $n \to \infty$. Passing to the limit in (\ref{e G>||}) and using the continuity of $G$, we get
for $\zeta_0 \in \TT$,
\begin{equation} \label{e G>|| T}
| G(\zeta_0) \psi |_\X \geq C^{-1} |\psi |_\X
\text{ and } \left| [G(\zeta_0)]^* \psi^* \right|_{\X^*} \geq C^{-1} |\psi^* |_{\X^*} .
\end{equation}
This implies that $\ker G (\zeta_0) = \{0_\X\}$ and $\ker [G (\zeta_0)]^* = \{0_{\X^*} \}$. The latter equality
also implies $ \overline{\ran G (\zeta_0)} = \X$ (see \cite[Lemma VI.2.8]{DSh58}). On the other hand, (\ref{e G>|| T})
yields $ \overline{\ran G (\zeta_0)} = \ran G (\zeta_0) $ (see \cite[Exercise VI.9.15]{DSh58}).
Hence $G (\zeta_0)$ is a self-bijection of $\X$. By (\ref{e G>|| T}),  $[G (\zeta_0)]^{-1}$ is bounded.

Thus, $G(\zeta)$ is boundedly invertible for all  $\zeta \in \overline{\DD (1)}$.
The set of $\zeta$, where $G (\zeta) $ is invertible with a bounded inverse, is open in $\DD (\varrho)$
(see \cite[Lemma VII.6.1]{DSh58}).
So $G $ is boundedly invertible on $\DD (R_1)$ with certain $R_1 \in (1,\varrho)$.
\end{proof}

\begin{proposition}\label{e Kl1=>RK>1}
Let (\ref{e Keq}) be defined on $\Bo^{\infty,0}$ and be UES in the resolvent matrix sense.
Then, the discrete function $K (\cdot) $ decays exponentially.
\end{proposition}

\begin{proof}
By the assumptions, $R [\wh X_K] > 1$ and $K \in \L (\Bo^{\infty,0},\X)$. Hence, $R [\wh K] \geq 1$. So for all $\zeta \in \DD (1)$,
(\ref{e K XK inv}) holds true and
\begin{equation} \label{e |X|<1+|K|}
\| \, [\wh X_K (\zeta)]^{-1} \|_{\X \to \X} \leq 1 + \| \wh K (\zeta ) \|_{\X \to \X} .
\end{equation}
On the other hand, for $\zeta \in \DD (1)$, $\psi \in \X$, and $\vphi :=  \{ \zeta^{-m} \psi \}_{m=-\infty}^0$, one has
$| \vphi |_{\Bo^{\infty,0}} \leq | \psi |_\X $ and
\[
| \wh K (\zeta ) \psi |_\X = \left| \sum_{j=0}^{+\infty} \zeta^j K (j) \psi \right|_\X =
| K \vphi |_\X \leq \| K \|_{\Bo^{\infty,0} \to \X} |\psi |_\X .
\]
So $ \sup_{\zeta \in \DD (1)} \| \wh K (\zeta ) \|_{\X \to \X} \leq  \| K \|_{\Bo^{\infty,0} \to \X} < \infty$.
Due to (\ref{e |X|<1+|K|}),
$ \sup_{\zeta \in \DD (1)} \| \, [\wh X_K (\zeta)]^{-1} \, \|_{\X \to \X} < \infty$.

Hence, we can apply Lemma \ref{l G bound inv}
to the $\L(X)$-valued function $\wh X_K $. We see that $\wh X_K$ is boundedly invertible on
$\DD (\varrho)$ with certain $\varrho > 1$.
By (\ref{e K XK inv}), the function $\zeta^{-1} \left( I_\X - [\wh X_K (\zeta)]^{-1} \right)$ is
an analytic continuation of $\wh K (\zeta)$ from $\DD (1)$ to $\DD (\varrho)$. Thus, $R [\wh K] \ge \varrho > 1$.
In other words, $K (\cdot) $ decays exponentially.
\end{proof}

\begin{proof}[Proof of Theorem \ref{t UEScr ga=0}]
\textbf{(iii) $\Rightarrow$ (ii)}. By (\ref{e R>1=UESXinX}), $ R [ \wh X_K ] > 1 $. By Proposition \ref{e Kl1=>RK>1},
$K(\cdot)$ decays exponentially.
From (\ref{e RwhK>1 equiv}), we see that there exists $ \ga_0 >0 $ such that (\ref{e Keq}) is defined on $\B^{p,\ga}$
for all $(p,\ga) \in [1,\infty] \times (0,\ga_0]$. So the assumption of Theorem \ref{t UEScr ga>0} is
fulfilled and the UE stability in the resolvent matrix sense implies the UE stability in $\X$ w.r.t. $\B^{p,\ga}$
for all $(p,\ga) \in [1,\infty] \times (0,\ga_0]$.

Proposition \ref{p st emb} proves the implication \textbf{(ii) $\Rightarrow$ (i)}.
For the implication \textbf{(i) $\Rightarrow$ (iii)} see the proof of Theorem \ref{t UEScr ga>0} (i) $\Rightarrow$ (iii).
\end{proof}

\subsection{Examples and remarks}
\label{ss Ex and Rem conv}

\begin{remark}
In the case $\X = \CC^n$ the proof of Proposition \ref{e Kl1=>RK>1}
can be simplified and Lemma \ref{l G bound inv}
is not needed. The reason is the obvious fact that,
 \begin{equation} \label{e Keq def = l1}
\text{for } \X = \CC^n, \text{ system (\ref{e Keq}) is
defined on } \Bo^{\infty,0} \text{ exactly when } K(\cdot) \in \ll^1 (\L(\X)).
\end{equation}
The proof of Proposition \ref{e Kl1=>RK>1} can be simplified in the following way (cf. \cite[Theorem 4]{EM96}): $K(\cdot) \in \ll^1 (\L(\X))$ yields that
$\wh K (\zeta)$ is convergent and uniformly bounded in $\overline{\DD (1)}$. Under the assumptions of Proposition \ref{e Kl1=>RK>1},
one can see that $\wh X (\zeta) $ is invertible in $\overline{\DD (1)}$, an so in an open neighborhood of $\overline{\DD (1)}$.
This means that $\wh K (\zeta)$ is convergent in an open neighborhood of $\overline{\DD (1)}$.
Thus, $K(\cdot)$ decays exponentially.
\end{remark}

For infinite-dimensional $\X$, \emph{the condition that (\ref{e Keq}) is
defined on $\Bo^{\infty,0}$ does not imply $K(\cdot) \in \ll^1 (\L(\X))$}. This is shown by the following example.

\begin{example} \label{ex c0}
Let $\X = c_0$, where $c_0$ is the usual Banach space of all convergent to zero sequences $a=[a_{[n]}]_{n=0}^{+\infty}=
[a_{[0]}, a_{[1]}, a_{[2]}, \dots ]$ of complex numbers.
Define an operator $K \in \L (\Bo^{\infty,0}, c_0 )$ by
 \begin{equation} \label{e K diag}
 K \vphi = \left[ (\vphi^{[0]})_{[0]} , (\vphi^{[-1]})_{[1]}, (\vphi^{[-2]})_{[2]}, \dots \right] .
\end{equation}
That is,
$  K \vphi = \sum_{m=-\infty}^0 K (-m) \vphi^{[m]} $ with $K(n) \in \L (c_0)$ defined by
$(K(n) a)_{[k]} = \delta_{nk} a_{[n]}$. Here $\delta_{nk}$ is  Kronecker's delta, and the infinite sum is understood
in the strong topology of $c_0$.

So $K$ is a bounded operator from $\Bo^{\infty,0}$ to $c_0$. System (\ref{e Keq}) is
defined on $\Bo^{\infty,0}$. On the other hand, $\| K(n) \|_{c_0 \to c_0} =1$ and so $ K(\cdot) \not \in \ll^1 $.
\end{example}

The following modification of the last example shows that the convolution system (\ref{e Keq}) can be defined on
$\B^{\infty,0}$ under weaker
assumptions on $K(j)$ than (\ref{e as K=sum}), and that such wider settings may sometimes be more natural.

\begin{example} \label{ex l inf}
Let $\X = \ll^\infty (\ZZ^+, \CC )$ be the Banach space of bounded sequences $[a_{[n]}]_{n=0}^{+\infty}$
of complex numbers. Let $K \in \L (\B^{\infty,0}, \X)$ be defined by (\ref{e K diag}).
Consider the corresponding discrete function $K(\cdot)$ (see (\ref{e def G dot}) for the definition).
Then $K(n)$ are defined as in Example \ref{ex c0}, but the infinite sum
$\sum_{m=-\infty}^0 K (-m) \vphi^{[m]} $  is divergent in the strong topology of $\X$  whenever
$\vphi \in \B^{\infty,0}$ does not satisfy $\lim_{j\to+\infty} (\vphi^{[-j]})_{[j]} = 0 $.
However, the representation $  K \vphi = \sum_{m=-\infty}^0 K (-m) \vphi^{[m]} $ still holds true for all
$\vphi \in \B^{\infty,0} $ if the sum is understood in the
weak$^*$ topology of $\X$.
(Concerning the representation $  K \vphi = \sum_{m=-\infty}^0 K (-m) \vphi^{[m]} $  in
the weak topology of $\X$, we refer a reader to the criterion of weak convergence \cite[Theorem 8.1.1]{S50}, \cite[Theorem IV.6.31]{DSh58}.)
\end{example}

\begin{remark}
Theorem \ref{t UEScr ga>0} remains valid if the additional assumption (\ref{e as K=sum}) is dropped
(i.e., it is valid for systems $x(n+1) = K x_n$ defined on $\B^{\infty,\ga}$).
Indeed, let the system $x(n+1) = K x_n$ be defined on $\B^{\infty,\ga}$ with $\ga>0$.
Then (\ref{e as K=sum}) holds for every $\vphi \in B^{p,\ga_1}$ with $\ga_1 \in (0,\ga)$.
In other words, on the narrower space $B^{p,\ga_1}$, the system $x(n+1) = K x_n$ takes the convolution form (\ref{e Keq}).
Therefore the equivalencies (ii) $\Leftrightarrow$ (iii) $\Leftrightarrow$ (iv) of Theorem \ref{t UEScr ga>0} hold true.
By Theorem \ref{e BPcr}, the equivalence (i) $\Leftrightarrow$ (iv) holds for the system $x(n+1) = K x_n$ on the original
phase space $\B^{\infty,\ga}$. This completes the proof.
\end{remark}

\begin{remark} \label{r comp ga>0}
Under the assumption $\sum_{j=1}^{+\infty} \| e^{j\ga} K (j)  \|_{\X \to \X} < \infty$,
system (\ref{e Keq}) was studied in \cite{FMN04_JMJ} in the settings of the phase space $\B^{\infty,\ga}$ with $\ga>0$.
In particular, the implication (ii) $\Rightarrow$ (iii) and the equivalence (i) $\Leftrightarrow$ (iii)
of Theorem \ref{t UEScr ga>0} were proved. The implication (iii) $\Rightarrow$ (ii) was proved for the case when all operators
$K(j)$ are compact. This compactness assumption is superfluous.
\end{remark}

\begin{remark} \label{r comp ga=0}
Under the assumption that $K(j)$ are compact operators and $\sum_{j=1}^{+\infty} \| K (j) \|_{\X \to \X} < \infty$,
it was shown in \cite[Theorem 4]{FMN04_JDEA} that (\ref{e Keq}) is UES in $X$ w.r.t. $\B^{\infty,0}$ exactly when (\ref{e Keq})
is uniformly asymptotically stable (UAS) and $K(\cdot) $ decays exponentially.
The result of \cite[Theorem 2]{FMN04_JDEA} and \cite[Theorem 1]{MN05} extend the criteria of UA stability of \cite{EM96}
to Banach space settings imposing the compactness assumption on $K(j)$ (see also \cite{MN08} for related results on positive systems).
In addition, \cite[Remark 1]{FMN04_JDEA} and \cite[Remark 1]{MN05} discuss the problem of removing the compactness assumption.
While our paper is concerned with UE stability, and so does not directly address the problem of \cite[Remark 1]{MN05}, Theorem \ref{t UEScr ga=0} and its proof may shed some light on this problem since they do not require
the compactness of operators $K(j)$ for a very kindred question of UE stability.
\end{remark}

The following example shows that the condition that system (\ref{e Keq})
is defined on $\Bo^{\infty,0}$ or $\B^{p,\ga}$ with $\ga>0$
cannot be dropped in Theorems \ref{t UEScr ga>0} and \ref{t UEScr ga=0}.
It also shows that the condition that $K(\cdot)$ decays exponentially
cannot be omitted in Corollary \ref{c BPcr ga=0}.

\begin{example}
Take $\X =\CC$ and $K(j) = - 2^{j+1}$. This leads to the system
\begin{equation} \label{e ex 2 j+1}
x(n+1) = - \sum_{j=0}^\infty 2^{j+1} x (n-j) ,
\end{equation}
which has the following properties.
\item[(i)] $R [\wh K] =1/2$ and system (\ref{e ex 2 j+1}) is not defined in $\B^{p,\ga}$ whenever $\ga > - \ln 2$ (for arbitrary $p$).
So (\ref{e ex 2 j+1}) is not UES in $\X$ w.r.t. these spaces.
\item[(ii)] $\IS_K \in \L (\ll^q )$ for each $1 \leq q \leq \infty$.
\item[(iii)] (\ref{e ex 2 j+1}) is UES in the resolvent matrix sense.

Assertion (i) is obvious. To check (ii) and (iii), note that
a solution $x (\cdot)$ to the nonhomogeneous system
$
x(n+1) = - \sum_{j=0}^n 2^{j+1} x (n-j) + f(n),  \ x(0)=0,
$
is given by
$x(1) = f(0)$ and $x(n) = f(n-1) -2 f(n-2)$, $n \geq 2$. In particular, $X_K (1) = -2$, $X_K (n) = 0$ for $n \geq 2$.
\end{example}

\section{Stability radii for various classes of perturbations}
\label{s pert}

\subsection{Definitions, a feedback scheme with delayed output}
\label{ss r def}

Let system (\ref{e Keq}) be defined on a phase space $\B$.
We consider linear time-invariant and time-varying structured perturbations of (\ref{e
Keq}) on $\B$. The structure of perturbations is described by the
operators $E \in \L (\B,\U_1)$ and $D \in \L (\U_2 , \X)$, where
$\U_{1,2}$ are auxiliary Banach spaces.

The perturbations of the following types are considered:
\begin{description}
\item[(Sc)] $  x(n+1)=\sum_{j=0}^{+\infty} \ K(j) \ x(n-j) + D \ \De \ E x_n $,
\item[(St)] $  x(n+1)=\sum_{j=0}^{+\infty} \ K(j) \ x(n-j) + D \ \De (n) \ E x_n $.
\end{description}
The corresponding \emph{disturbance (or unknown feedback) mappings} $\De$ and $\De (n)$ have
the following properties:
\begin{description}
\item[(Pc)] $ \De  \in \L (\U_1,\U_2)$ is a time-invariant disturbance operator.
\item[(Pt)] $ \De (\cdot) \in \ll^\infty \left( \L (\U_1,\U_2) \right) $ is an operator-valued function describing time-varying linear disturbances.
\end{description}

The perturbed systems (Sc)-(St) can be interpreted as \emph{feedback systems with delayed output}, see Fig.\ref{fig1}.
Note that the output $y (n) = E x_n$ depends on the prehistory $x_n = \{ x (n+m) \}_{m=-\infty}^{0}$ (delayed output)
and that the input $v(n)$ is connected with the output by $ v(n) = N(n) y (n)$, where
an unknown operator $N (n)$ of feedback is given by $\De$ or $\De (n)$, respectively.
$U_2$ ($U_1$) turns into the \emph{input  (resp., output)} space.

%

\begin{definition} \label{d IO K}
\emph{The input-output operator} $\IO_{K}: \Ssp_+ (\U_2) \to \Ssp_+ (\U_1)$
corresponding to Fig.\ref{fig1} is defined by $\IO_K : v (\cdot) \to y (\cdot)$,  where
$y (n) = E \left(\{ x (n+m) \}_{m=-\infty}^0 \right)$ and $x (\cdot)$ is the solution to the
system
\begin{equation} \label{e Keq in}
x(n+1)=\sum_{j=0}^{n} \ K(n-j) \ x(j) + D v (n) , \quad n\geq 0, \ \
\ x_0 =0_\B .
\end{equation}
\end{definition}

\begin{definition} \label{d rUES}
\emph{The (UE) stability radius} of (\ref{e Keq})  w.r.t. perturbations of the
structure $(D,E)$, the disturbances of the class (Pc), and the phase
space $\B$ is defined by
\begin{equation*} \label{e def r c}
\r_\c (K; D , E; \B)  =  \inf \{ \| \De \|_{\U_1 \to \U_2} \  : \ \De \in \L (\U_1,\U_2), \text{ and \textrm{(Sc)}
is not UES } \}.
\end{equation*}
Usually, we will drop $K$ in this notation.
The stability radius $\r_\t (D , E; \B)$ w.r.t. the
disturbances of the class (Pt) is defined in the
analogous way
\begin{equation*} \label{e def r t}
\r_\c (D , E; \B)  =  \inf \{ | \De (\cdot) |_{\ll^\infty } \  : \ \De (\cdot) \in \ll^\infty \left( \L
(\U_1,\U_2) \right), \text{ and \textrm{(St)} is not UES } \}.
\end{equation*}
\end{definition}

Identifying an operator $\Delta \in \L (\U_1,\U_2)$ with the constant discrete
function $\{ \Delta, \Delta, \cdots \}$, one gets a norm-preserving embedding
$\L (\U_1,\U_2) \subset \ll^\infty \left( \L (\U_1,\U_2) \right)$.
This implies
\begin{equation} \label{e rc>rt>rnt}
 \r_\c (D , E; \B) \ge \r_\t (D , E; \B)  .
\end{equation}

It follows from Proposition \ref{p st emb} that, for phase spaces $\B$ and $\B_1$,
\begin{equation} \label{e r emb}
\text{continuous embedding } \B_1 \subset \B \ \Longrightarrow \
\r_i (D , E; \B_1) \ge \r_i (D , E;\B) , \ \ \text{ where } i = \c, \t .
\end{equation}

\subsection{Main results: stability radii in $\Bo^{\infty,0}$ and in $\B^{p,\ga}$ with $\ga > 0$}
\label{ss radii est}

By the operator $E \in \L (\B,\U_1)$, we define a function $E(\cdot)$ and the associated Z-transform $\wh E(\cdot)$ in the way shown by (\ref{e def G dot}) .

\begin{theorem} \label{t r ga>0}
Let $\ga >0$ and $1 \leq q \leq \infty$. Let $\B=\Bo^{\infty,\ga}$ or $\B=\B^{p,\ga}$ with $1 \le p \le \infty$.
Let (\ref{e Keq}) be UES in $\X$ w.r.t. $\B$.
Then $\IO_K \in \L \left( \ll^q (\U_2), \ll^q (\U_1)\right) $ and
\begin{multline}
 \Bigl( \max_{|\zeta|=1} \| \wh E (\zeta)  [I_\X - \zeta  \wh K (\zeta) ]^{-1} D \|_{\U_2 \to \U_1} \Bigr)\!^{-1} =
 \r_\c (D , E; \B) \ge \r_{\t} (D , E; \B) \ge
 \| \IO_K \|_{ \ll^q (\U_2) \to \ll^q (\U_1)}^{-1}  .\label{e rGAS>}
\end{multline}

If, additionally, $q=2$ and $\X$, $\U_1$, $\U_2$ are Hilbert spaces,
then (\ref{e rGAS>}) holds with equalities, i.e.,
\begin{multline}
 \Bigl( \max_{|\zeta|=1} \| \wh E (\zeta)  [I_\X - \zeta  \wh K (\zeta)]^{-1} D \|_{\U_2 \to \U_1} \Bigr)\!^{-1} =
 \r_\c (D , E; \B)  = \r_\t (D , E; \B)
 =
 \| \IO_K \|_{\ll^2 (\U_2) \to \ll^2 (\U_1)}^{-1}  .
\label{e r=}
\end{multline}
\end{theorem}

\begin{remark} Let $\X$, $\U_1$, and $\U_2$ be Hilbert spaces, but $p \neq 2$. 
Then, the phase space $\B$
is not a Hilbert space, but, according to the theorem, equalities (\ref{e r=}) still hold (cf.
\cite[Corollary 4.5]{WH94}). The proof of this part of the theorem requires an additional step.
\end{remark}

The proof is given in Section \ref{ss reduction}.

\emph{Let us turn to stability radii in the non-fading phase space
$\Bo^{\infty,0}$}. If $D = 0_{\U_2 \to \X}$ or $E = 0_{\B \to
\U_1}$, the answer is trivial and not interesting: all the stability
radii are equal to $\infty$. In the case $D \neq 0_{\U_2 \to \X}$,
it occurs that the stability radii may be positive only if the
operator $E$, which is initially assumed to be in $\L
(\Bo^{\infty,0},\U_1)$, satisfies an additional condition.

\begin{proposition} \label{p r=0}
Let (\ref{e Keq}) be UES in $\X$ w.r.t. $\Bo^{\infty,0}$ and $D \neq
0_{\U_2 \to \X}$. If $\r_\c (D , E; \Bo^{\infty,0}) > 0$,
then the discrete function $E (\cdot)$ decays exponentially.
\end{proposition}

The proof is not too long and illustrates well the use of Theorem
\ref{t UEScr ga=0} (i) $\Leftrightarrow$ (ii).

\begin{proof}
Assume $\r_\c (D , E; \Bo^{\infty,0}) > 0$. Then for any $\De_0
\in \L (\U_1,\U_2) $ there exists small $\ep = \ep (\De_0) > 0 $ such that the time-invariant system (Sc) with $\De =
\ep \De_0$ is UES in $\X$ w.r.t. $\Bo^{\infty,0}$. By Theorem \ref{t
UEScr ga=0} (i) $\Leftrightarrow$ (ii), there exists $\ga >0$ such
that both systems (\ref{e Keq}) and (Sc) are defined on
$\B^{\infty,\ga}$. Hence,
\begin{equation} \label{e ext DDe0E}
\text{the operator } D \De E \text{ can be extended by continuity to } \B^{\infty,\ga} .
\end{equation}

Assume now that $E (\cdot)$ does not decay exponentially. 
Then, there exist an increasing sequence $n_k$ such that
$\lim_{k\to \infty} \| e^{\ga n_k} E (n_k) \|_{\X \to \U_1} = \infty$. Choose $\psi (k) \in \X$
with the properties $| \psi (k) |_\X = 1$ and $|E (n_k) \psi (k)|_{\U_1} > \frac 12 \ \| E (n_k) \|_{\X \to \U_1}$
for all $k \in \NN$.
Consider $\vphi (k) \in \Bo^{\infty,0}$ defined by
\[
\vphi^{[m]} (k) = \left\{ \begin{array}{ll} e^{\ga n_k} \psi (k) & \text{ if } m = -n_k \\
0_\X, & \text{otherwise}
\end{array} \right. \ .
\]
Then
\begin{equation} \label{e phi=1 Ephi to inf}
 | \vphi (k) |_{\B^{\infty,\ga}}  = 1 \ \ \text{ for all } k, \ \ \ \
 \text{ but } \ \
\lim_{k\to \infty } | E \vphi (k) |_{\U_1} = \infty .
\end{equation}
Indeed, $| E \vphi (k) |_{\U_1} = \left| \sum_{j=0}^{\infty} E (j) \vphi^{[-j]} (k) \right|_{\U_1} =
e^{\ga n_k} | E(n_k) \psi (k)|_{\U_1}$ and so
\[
| E \vphi (k) |_{\U_1} \geq \frac 12 \ e^{\ga n_k} \| E (n_k) \|_{\X \to \U_1} \to \infty.
\]

By (\ref{e phi=1 Ephi to inf}) and the uniform boundedness principle, there exists
$u^* \in \U_1^*$ such that $\left| u^* \left( E \vphi (k) \right) \right| \to \infty$.
Let $u_2 \in \U_2$ be such that $\psi = D u_2 \neq 0_\X$.
Consider an operator $\De_0 \in \L (\U_1,\U_2) $
defined by $\De_0 u_1 = u^* (u_1) \ u_2$. Then
\[
| D \De E \vphi (k) |_{\X} = \ep (\De_0) \, | D \De_0 E \vphi (k) |_{\X} =
\ep (\De_0) \ | \psi |_\X  \ | u^* (E \vphi (k)) |  \to \infty .
\]
This contradicts (\ref{e ext DDe0E}).
\end{proof}

\begin{theorem} \label{t r Bo}
Let $1\leq q \leq \infty$. Let (\ref{e Keq}) be UES in $\X$ w.r.t. $\Bo^{\infty,0}$, and let $E (\cdot)$ decay exponentially.
Then
$\IO_K \in \L \left( \ll^q (\U_2), \ll^q (\U_1)\right) $ and
\begin{multline}
 \Bigl( \max_{|\zeta|=1} \| \wh E (\zeta)  [I_\X - \zeta  \wh K (\zeta)]^{-1} D \|_{\U_2 \to \U_1} \Bigr)\!^{-1} =
 \r_\c (D , E; \Bo^{\infty,0}) \ge \r_{\t} (D , E; \Bo^{\infty,0}) \ge  \| \IO_K
\|_{ \ll^q (\U_2) \to \ll^q (\U_1)}^{-1}  .\label{e rUES> Bo}
\end{multline}

If, additionally, $q=2$ and $\X$, $\U_1$, $\U_2$ are Hilbert spaces, then
(\ref{e rUES> Bo}) holds with equalities.
\end{theorem}

The proof is given in Section \ref{ss p t r Bo}.

The assumption that $E (\cdot)$ decays exponentially is always satisfied in the important case
when $E$ defines \emph{perturbations with bounded delay},
i.e., when $E (n) = \Oz_\X$ for $n$ large enough.

\section{Unstructured perturbations and the norm of the input-state operator} \label{s unstr}

For a fixed phase space $\B$, consider the perturbed system
\begin{equation} \label{e unstr}
x(n+1)=\sum_{j=0}^{+\infty} \ K(j) \ x(n-j) + \Nper \! (n) x_n  ,
\end{equation}
where the restrictions similar to (Pc), (Pt) are imposed on the disturbance mappings $\Nper (n) \in \L(\B,\X)$,
$n \in \ZZ^+$.
That is, $\Nper (n)$ is supposed to be either
time invariant $ \Nper (n ) =  \De $ with $\De  \in \L (\B,\X)$ or time-varying
$ \Nper ( n ) =  \De (n) $ with $\De (\cdot) \in \ll^\infty \left( \L (\B,\X) \right)$.

The definition of the corresponding stability radii $\r_i (K; \B)$ can
be given in the way similar to that of Definition \ref{d rUES}, or,
alternatively, one can notice that the perturbed systems under consideration are particular cases of
(Sc),(St) with the very simple choice of the perturbation structure
\[
\U_1 = \B, \ \ \ E=I_\B, \ \ \  \U_2 = \X, \ \ \ D=I_\X.
\]
So, \emph{the unstructured stability radii} can be defined by
\[
\r_i (K; \B) := \r_i (K; I_\X , I_\B; \B)
, \ \ \ i=\c,\t.
\]

The input-output operator $\IO_K $ (see Definition \ref{d IO K})
turns into the unstructured input to prehistory of state operator, i.e.,
$\IO_K : v (\cdot) \to x_\bullet $,
where $x_n = \{ x (n+m) \}_{m=-\infty}^0$ and $x (\cdot)$ is the solution to the
system
$x(n+1)=\sum_{j=0}^{n} \ K(n-j) \ x(j) + v (n)$, $x_0 =0_\B$.

The discrete function $E (\cdot)$ for $E=I_{\B^{p,\ga}}$ is $E (n) = P_{-n}^\T $, see (\ref{e def G dot}).

\emph{Let us start with the unstructured radii in the case of the non-fading phase space $\B = \Bo^{\infty,0}$.}
Since $\| E (n) \|_{\X \to \Bo^{\infty,0}} = \| P_{-n}^\T \|_{\X \to \Bo^{\infty,0}} = 1$, we see that the discrete function
$E (\cdot)$ does not decay exponentially. By Proposition \ref{p r=0}, $ \r_\c (K; \Bo^{\infty,\ga}) =0 $.
Due to (\ref{e rc>rt>rnt}), we get the following.

\begin{corollary} \label{c unstr Bo}
$ \r_\c (K; \Bo^{\infty,0}) = \r_\t (K; \Bo^{\infty,0}) = 0$.
\end{corollary}

\emph{Consider fading-phase spaces}, i.e., the case when $\ga>0$ and $\B = \B^{p,\ga}$ or $\B
= \Bo^{\infty,\ga}$ .

\begin{corollary} \label{c unstr}
Let $\ga >0$ and $1 \le p \le \infty$. Let $\B=\B^{p,\ga}$ or $\B=\Bo^{\infty,\ga}$ (in the latter case it is assumed that
$p=\infty$).
Let (\ref{e Keq}) be UES in $\X$ w.r.t. $\B$.
Then $\IS_K \in \L \left( \ll^p (\X) \right) $ and
\begin{multline}
(1-e^{-p\ga})^{1/p}  \left( \max_{|\zeta|=1} \| \, [I_\X - \zeta  \wh K (\zeta)]^{-1} \|_{\X \to \X} \right)^{-1} =
 \r_\c (K; \B) \ge \r_{\t} (K; \B) \ge \\ \ge
\left( 1-e^{-p\ga} \right)^{1/p}
\| \IS_K \|_{ \ll^p (\X) \to \ll^p ( \X) }^{-1}  ,
\label{e r> unstr}
\end{multline}
where  $e^{-p\ga} $ and $1/p$ have to be understood as zero when
$p=\infty$.

If, additionally,  $p=2$ and $\X$ is a Hilbert space, then
(\ref{e r> unstr}) holds with equalities.
\end{corollary}

\begin{proof}
Since $E (n) = P_{-n}^\T $, we see that
$\left( \wh E (\zeta)
\psi \right)^{[m]} = \zeta^{-m} P_{-n}^\T \psi $ for $m \in
\ZZ^-$ and $\psi \in \X$.
Taking $\zeta \in \{ z \in \CC : |z| =1 \}$ and $v \in \X$, we have for $p <\infty$
\begin{eqnarray*}
\left| \wh E (\zeta)  [I_\X - \zeta  \wh K (\zeta)]^{-1} D v
\right|_{\B^{p,\ga}}^p = \!\! \sum_{m=-\infty}^{0} e^{pm\ga} \left| [I_\X -
\zeta  \wh K (\zeta)]^{-1} v \right|_\X^p  = (1-e^{-p\ga})^{-1}  \left| [I_\X -
\zeta \wh K (\zeta)]^{-1} v \right|_\X^p ,
\end{eqnarray*}
and $\left| \wh E (\zeta)  [I_\X - \zeta  \wh K (\zeta)]^{-1} D v
\right|_{\B_{(0)}^{\infty,\ga}} = \left| [I_\X - \zeta  \wh K (\zeta)]^{-1} v
\right|_\X $ when $p=\infty$. Theorem \ref{t r ga>0} gives
\begin{multline*}
(1-e^{-p\ga})^{1/p}  \left( \max_{|\zeta|=1} \| (I_\X - \zeta  \wh K (\zeta))^{-1} \|_{\X \to \X} \right)^{-1} =
 \r_\c (K; \B) \ge \r_{\t} (K; \B) \ge
 \| \IO_K \|_{ \ll^q (\X) \to \ll^q (\B)}^{-1} ,
\label{e r> IO unstr}
\end{multline*}
where $q$ can be chosen arbitrary in the range $1 \le q \le \infty$.

With this extremely simple choice of the structure, the operator $\IO_K$ can be
expressed through the unstructured input-state operator $\IS_K$.
Indeed,
\begin{equation} \label{e IO=(IS)}
(\IO_K v) (n) = x_n = \{ \dots, x(n-1) , x(n) \} = \{  \dots , (\IS_K v) (n-1) , (\IS_K v) (n) \} .
\end{equation}
Put $q=p$. Then, in the case $p<\infty$,
\begin{equation} \label{e IO=IS}
\| \IO_K \|_{ \ll^p (\X) \to \ll^p ( \B^{p,\ga} )} = \left( 1-e^{-p\ga} \right)^{-1/p}
\| \IS_K \|_{ \ll^p (\X) \to \ll^p ( \X) } .
\end{equation}
In fact, (\ref{e IO=(IS)}) implies
\begin{eqnarray*}
| \IO_K v |_{\ll^p ( \B^{p,\ga} )}^p & = & \sum_{n=1}^{+\infty} |
x_n |^p_{\B^{p,\ga}} = \sum_{n=1}^{+\infty} \sum_{j=0}^{n-1}
| e^{-j\ga} x (n-j) |_\X^p = 
\left( 1-e^{-p\ga} \right)^{-1} \sum_{k=1}^{+\infty} | (\IS_K v) (k) |_\X^p .
\end{eqnarray*}
When $p=\infty$, we obviously have
$\| \IO_K \|_{ \ll^\infty (\X) \to \ll^\infty ( \B^{\infty,\ga} )} =
\| \IS_K \|_{ \ll^\infty (\X) \to \ll^\infty ( \X) } $
\end{proof}

\begin{corollary} \label{c |I-whK|=|IS|}
Assume that $1 \le p \le \infty$, $\IS_K \in \L (\ll^p (\X))$, and $R [\wh K] > 1$.
Then $I_\X - \zeta \wh K (\zeta) $ is boundedly invertible for all $\zeta \in \overline{\DD(1)}$ and
\begin{equation}
 \max_{|\zeta|=1} \| \, [I_\X - \zeta  \wh K (\zeta)]^{-1} \|_{\X \to \X} =
 \max_{|\zeta|\leq 1} \| \, [I_\X - \zeta  \wh K (\zeta)]^{-1} \|_{\X \to \X}
 \le
\| \IS_K \|_{ \ll^p (\X) \to \ll^p ( \X) } .
\label{e |I-whK|<|IS|}
\end{equation}

If, additionally,  $\X$ is a Hilbert space and $p=2$, then
the equality hold in (\ref{e |I-whK|<|IS|}).
\end{corollary}

\begin{proof}
It follows from $R [\wh K] > 1$, that (\ref{e Keq}) is defined on $B^{p,\ga}$ for certain $\ga>0$.
By Theorem \ref{t UEScr ga>0}, the assumption $\IS_K \in \L (\ll^p )$ implies the UE stability of (\ref{e Keq}) in $\X$ w.r.t.
$B^{p,\ga}$, and also implies, the bounded invertibility of $I_\X - \zeta \wh K (\zeta) $
for all $\zeta  \in \overline{\DD (1)} $.
Now (\ref{e |I-whK|<|IS|}) follows from Corollary \ref{c unstr} and the maximum modulus principle.
\end{proof}

\section{Proofs of Theorems \ref{t r ga>0} and \ref{t r Bo}: two reductions}
\label{s proof radii}

\subsection{Stability radii for first order systems}
\label{ss Stab first ord}

First, we consider stability radii for a \emph{linear first
order time-varying} system
\begin{equation}
w (n+1) = \A (n) w (n) ,  \quad n \in \ZZ^+,  \label{e A s}
\end{equation}
where $ w \in \Ssp_+ (\W)$, $\A (n) \in \L (\W)$ for all $n$, and $\W$ is a certain
Banach space.

Let $\U_1$, $\U_2 $ be auxiliary Banach spaces. Let  $\wt E \in \L
(\W,\U_1)$ and $\wt D \in \L (\U_2 , \W)$. Following \cite{WH94},
consider two classes of structured perturbations for (\ref{e
A s})
\begin{description}
\item[(FOSc)] $ w(n+1) = \A (n) w (n) + \wt D \ \De \ \wt E w (n) $,
\item[(FOSt)] $ w(n+1) = \A (n) w (n) + \wt D \ \De (n) \ \wt E w(n) $,
\end{description}
where the disturbance mappings  $\De$ and $\De (n)$   have
the properties (Pc) and (Pt) of Section \ref{ss r def}, respectively.

The UE stability for first order
systems is defined as usual (i.e., the norm $| \cdot |_\W$
replaces both the norms $| \cdot |_\X$ and $| \cdot |_\B$ in Definition \ref{d UESn}, see e.g. \cite{P87,SS04,BK10}).

The stability radii of (\ref{e A s})  w.r.t. perturbations of the
structure $(\wt D, \wt E)$ and the disturbances classes (Pc) and (Pt) are
defined by
\begin{eqnarray*} \label{e def r fo c}
\r_\c (\A; \wt D , \wt E)  & := & \inf \{ \| \De \|_{\U_1 \to \U_2} \  \ : \ \De \in
\L(\U_1,\U_2), \text{ and (FOSc)  is not UES } \}, \\
\r_\t (\A; \wt D , \wt E) & := & \inf \{ | \De (\cdot) |_{\ll^{\infty}} \  \ : \ \De (\cdot) \in
\ll^\infty \left(\L(\U_1,\U_2) \right), \text{ and (FOSt)  is not UES } \}. \label{e def r fo t}
\end{eqnarray*}

The unstructured input-state
operator $\IS_\A : \Ssp_+ (\W) \to \Ssp_+ (\W)$ associated with (\ref{e A s}) is defined by $(\IS_\A f) (\cdot)= w (\cdot)$, where $w=w(\cdot)$
is the solution to the nonhomogeneous system
\begin{equation*} \label{e Asys f}
w(n+1)=\A (n) w (n) + f(n) , \quad n\geq 0, \ \ \
w (0) =0_\W .
\end{equation*}
The input-output operator $\IO_{\A}: \Ssp_+ (\U_2) \to \Ssp_+ (\U_1)$ corresponding to (\ref{e A s})
and the perturbation structure $(\wt D, \wt E)$ is defined analogously to Definition \ref{d IO K}, i.e.,
$(\IO_\A v) (n) := \wt E w(n)$, where $w (\cdot)$ is the solution to $w(n+1)=\A (n) w (n) + \wt D v (n) ,$ $\ w(0) =0_\W $.

For any $1 \le q \le \infty$, the following criterion of Bohl-Perron type holds
\begin{equation} \label{e BPcr fo}
\text{system (\ref{e A s}) is UES } \ \
\Longleftrightarrow  \ \ \IS_\A \in \L \left( \ll^{q} (\W) \right) ,
\end{equation}
see \cite{P87,AVM96,SS04} (and also discussion in \cite[Sect.~2.2]{BK10}) for a stronger version of this result.
In particular, UE stability of (\ref{e A s}) implies $\IO_\A \in \L \left( \ll^q (\U_2) , \ll^q (\U_1) \right)$.

\begin{theorem}\label{t r>L fo}
Suppose (\ref{e A s}) is UES and $1 \leq q \leq \infty$.
Then
\[
\r_\c (\A; \wt D , \wt E) \ge
\r_\t (\A; \wt D , \wt E)  \ge  \| \IO_\A \|_{\ll^q (\U_2) \to \ll^q (\U_1)}^{-1} .
\]
\end{theorem}

In the case $q=2$, this result is known, see \cite[Theorem 3.1]{WH94}.
The proof of \cite{WH94} can be modified to cover $1 \le q < \infty$ if one uses \cite[Theorem 4.2]{PR84}
instead of \cite[Proposition 2.4 (iv)]{WH94}. However this proof does not work when $q=\infty$.
The following proof, which includes the case $q=\infty$, is based on the Bohl-Perron criterion (\ref{e BPcr fo}).

\begin{proof}[The proof of Theorem  \ref{t r>L fo}.]
It is enough to prove the second inequality.

For a function $\De : \ZZ^+ \to \L (\U_1,\U_2) $, we define the operator $\M_{\De} : \Ssp_+ (\U_1) \to \Ssp_+ (\U_2)$
of multiplication on $\De (\cdot)$ by $(\M_{\De} y) (n) = \De (n) y (n)$.
Similarly, by $\M_{\wt D}$ ($\M_{\wt E}$), the operator of multiplication on $\wt D$ (resp., $\wt E$) in the space $\Ssp_+ (\U_2)$
(resp.,  $\Ssp_+ (\W)$) is denoted,
\[
(\M_{\wt D} v) (n) = \wt D v (n) , \ \ v:\ZZ^+ \to \U_2 , \ \ \ \ (\M_{\wt E} w) (n) = \wt E w (n) , \ \ w:\ZZ^+ \to \W .
\]
The unstructured input-state
operator $\IS_\A$ and the input-output operator $\IO_\A$ are connected by
\begin{equation}
\IO_\A  = \M_{\wt E} \IS_\A \M_{\wt D} .
\end{equation}

Suppose (\ref{e A s}) is UES. Then $\IS_\A \in \L \left( \ll^q (\W) \right) $ and
$\IO_\A \in \L \left(\ll^q (\U_2),\ll^q (\U_1)\right) $. Assume that
$| \De (\cdot) |_{\ll^\infty} < \| \IO_\A \|_{\ll^q (\U_2) \to \ll^q (\U_1)}^{-1} $.
Since $| \De (\cdot) |_{\ll^\infty} = \| \M_{\De} \|_{\ll^q (\U_1) \to \ll^q (\U_2)}$,
we have
\[
\| \M_{\De} \IO_\A \|_{{\ll^q (\U_2) \to \ll^q (\U_2)}} < 1.
\]
This allows one to define an operator $\wt \IS \in \L \left(\ll^q (\W)\right)$ by
\[
\wt \IS :=  \IS_\A \M_{\wt D} \left[\sum_{j=0}^{+\infty} (\M_{\De} \IO_\A)^j \right] \M_{\De} \M_{\wt E} \IS_\A + \IS_\A .
\]
This definition implies $\wt \IS =  \IS_\A (\M_{\wt D} \M_{\De} \M_{\wt E} \wt \IS + I_{\ll^q } )$.
So $w(\cdot) = (\wt \IS f) (\cdot)$ is the solution to the system
\[
w(n+1) = \A (n) w (n) + \wt D \ \De (n) \ \wt E w(n) + f (n), \ n \geq 0 , \ \ w(0)=0 .
\]
In other words, $\wt \IS$ is the unstructured input-state operator of the perturbed system (FOSt).
Since $\wt \IS $ is bounded in $\ll^q (\W)$, the Bohl-Perron criterion (\ref{e BPcr fo}) implies that
the perturbed system (FOSt) is UES. This completes the proof.
\end{proof}

Now we apply the above theorem to strengthen some of the results of \cite{WH94} on \emph{linear first
order time-invariant systems} so that they fit to our needs.

When $\A (n) = A $ for all $n$ with $A \in \L (\W)$, system (\ref{e A s}) takes the form
\begin{equation}
w (n+1) = A w (n) ,  \quad n \in \ZZ^+.  \label{e Asys}
\end{equation}
The corresponding input-output operator and stability radii are denoted by $\IO_{A}$ and $\r_i (A; \wt D , \wt E) $,
$i=\c,\t$, respectively.

\begin{theorem}[cf. \cite{WH94}] \label{t r fo}
Suppose (\ref{e Asys}) is UES and $1 \leq q \leq \infty$.
Then $\IO_A \in \L \left( \ll^q (\U_2) , \ll^q (\U_1) \right)$ and the following statements hold.
\item[(i)]
$ \left( \max_{|\la|=1} \| \wt E (\la I_\W -A )^{-1} \wt D \|_{\U_2 \to \U_1} \right)^{-1} =
\r_\c (A; \wt D , \wt E) \ge
\r_\t (A; \wt D , \wt E)  \ge  \| \IO_A \|_{\ll^q (\U_2) \to \ll^q (\U_1)}^{-1} $.

\item[(ii)]
If, additionally,  $\W$, $\U_1$, $\U_2$ are Hilbert spaces (and $q=2$), then
\begin{multline*}
\Bigl( \max_{|\la|=1} \| \wt E (\la I_\W -A )^{-1} \wt D \|_{\U_2 \to \U_1} \Bigr)^{-1} = \r_\c (A; \wt D , \wt E)=
\r_\t (A; \wt D , \wt E) =
 \| \IO_A \|_{\ll^2 (\U_2) \to \ll^2 (\U_1)}^{-1}  .
\end{multline*}
\end{theorem}

Recall that time-invariant system (\ref{e Asys}) is UES if and only
if the spectral radius of $A$ is less than $1$.
Statement (ii) and, in the case $q=2$, statement (i) of this theorem
follows immediately from a combination of \cite[Corollary 4.5 and Proposition 5.3]{WH94}.
Statement (i) for $q\neq 2$ is a combination of the above mentioned results of \cite{WH94}
with Theorem \ref{t r>L fo}.

\subsection{The proof of Theorem \ref{t r ga>0}: reduction of order}
\label{ss reduction}

Let $\ga \in \RR$ and let $\B = \B^{p,\ga}$ or $\B = \Bo^{\infty,\ga}$. We want to write
the Volterra convolution system (\ref{e Keq}) defined on the phase space $\B$
and the perturbed  systems (Sc)--(St) in the form of first order systems.

Recall that system (\ref{e Keq}) can be written in the form $x (n+1) = K x_n$, where $K \in \L ( \B, \X)$.
Define the backward shift operator $\ShB$ in $\X^{\ZZ^-}$ (and so in all the phase spaces) by
\begin{eqnarray*} \label{e Sh}
 (\ShB \vphi)^{[m]}  :=  \left\{
\begin{array}{ll}
0_\X , &  m = 0, \\
\vphi^{[m+1]}  , & m \leq -1
\end{array} \right. .
\end{eqnarray*}
Then the first order system (\ref{e Asys}) with
\begin{equation}
A := P_0^\T K + \ShB \in \L (\B) \label{e Ared}
\end{equation}
and  $\W= \B$ is associated with system (\ref{e Keq}) in the sense that
\begin{equation} \label{e x_n = wn}
x_n (\tau, \vphi) = w (n, \tau, \vphi) ,
\end{equation}
where $w ( \cdot ,\tau, \psi)$ is a unique solution to system (\ref{e Asys})
satisfying the initial condition $w (\tau) = \psi$.
The operator $A$ can be written in the form of matrix with $\L(\X)$-entries:
\[
A = \left( \begin{array}{ccccccc}
K(0) & K(1) & K(2) & \dots & K(j-1) & K(j)  & \dots \\
I_{\X} & 0_\X & 0_\X & \dots & 0_\X  & 0_\X & \dots \\
0_{\X} & I_\X & 0_\X & \dots & 0_\X  & 0_\X & \dots \\
\dots & \dots & \dots & \dots & \dots  & \dots \\
0_{\X} & 0_\X & 0_\X & \dots & 0_\X  & 0_\X & \dots \\
0_{\X} & 0_\X & 0_\X & \dots & I_\X  & 0_\X & \dots \\
\dots & \dots & \dots & \dots & \dots  & \dots & \dots
 \end{array} \right) .
\]

Given the structure $\{ E, D \}$ of perturbations (Sc)-(St), we define the structure $\{ \wt E, \wt D \}$ of perturbations (FOSc)-(FOSnt) putting
\begin{equation}
\wt E := E, \ \ \ \wt D := P_0^\T D . \label{e EDred}
\end{equation}
Then solutions of  (Sc)-(St) and of (FOSc)-(FOSt), resp., are also
connected by (\ref{e x_n = wn}). Moreover, the input-output
operators are identical
\begin{equation} \label{e IO=}
\IO_A = \IO_K .
\end{equation}

The above procedure may be considered as
a generalization to systems with infinite delay of the phase-space method, which is well developed for systems with bounded delay,
see e.g. \cite{PH93}.

\begin{proposition} \label{p r=r new}
Let $\ga >0$. Let  $\B = \B^{p,\ga}$ or $\B=\Bo^{\infty,\ga}$. Assume (\ref{e
Ared}) and (\ref{e EDred}).
Then:
\item[(i)] system (St) is UES in $\X$ w.r.t. $\B$ if and only if
(FOSt) is UES,
\item[(ii)] $\r_i (K; D , E; \B)  = \r_i (A; \wt D , \wt E)$, $i = \c,\t$.
\end{proposition}

Statement (i) is proved in our previous paper \cite[Proposition 3.12 and Sec.~7.2]{BK10}, statement (ii) follows from (i).

\begin{proposition} \label{p resolvent}
Assume (\ref{e Ared}) and (\ref{e EDred}). Assume that both $\zeta^{-1} I_\B
-A $ and $ I_\X - \zeta  \wh K (\zeta) $ are boundedly invertible
for a certain $\zeta \in \CC \setminus \{ 0\}$. Then for any $v
\in \U_2$:
\item[(i)]
$ (\zeta^{-1} I_\B -A )^{-1} \wt D v = \left\{ \zeta^{-m+1}
[I_\X - \zeta  \wh K (\zeta)]^{-1} D v \right\}_{m=-\infty}^{0} $,
\item[(ii)] if additionally $\zeta \in \DD (e^{\ga}) \setminus \{0\}$ (with $\ga$ from the definition of $\B$), then
\[
 \wt E (\zeta^{-1} I_\B -A )^{-1} \wt D = \zeta \wh E (\zeta)
[I_\X - \zeta  \wh K (\zeta)]^{-1} D .
\]
\end{proposition}

\begin{proof}
\textbf{(i)} Due to  (\ref{e Ared}) and (\ref{e EDred}), the
equality $ (\zeta^{-1} I_\B -A ) \vphi =  \wt D v$ can be
rewritten as the system
\begin{eqnarray*}
\zeta^{-1} \vphi^{[0]} & - & \sum_{j=0}^{+\infty} K (j) \vphi^{[-j]}  = D v \ \ (\text{for } m=0), \\
\zeta \vphi^{[m+1]} & = & \vphi^{[m]}  \ \ \text{ for } m = -1,-2, \dots \ .
\end{eqnarray*}
This leads to $\vphi^{[m]} = \zeta^{-m} \vphi^{[0]}$, $m \le -1$,
and in turn to $ \vphi^{[0]} = \zeta [I_\X - \zeta  \wh K
(\zeta)]^{-1} D v$. So $\vphi$ is found and gives (i).

\textbf{(ii)} \emph{First,} take a simplified point of view that
\begin{equation} \label{e Erep}
E \vphi = \sum_{j=0}^{+\infty} E
(j) \vphi^{[-j]}
\end{equation}
for all $\vphi \in \B$ (this holds for each $E \in \L (\B,\U_1)$ in all the phase spaces except $\B^{\infty,\ga}$, see Remark \ref{r p=inf repr}).
Since $E \in \L (\B, \U_1)$ (where $\B = \B^{p,\ga}$ or $\B = \Bo^{\infty,\ga}$),
we see that $\wh E (\zeta) $ is defined for $|\zeta|<e^\ga$.
Since $\wt E  =E$, we get using (i) that
\[
\wt E (\zeta^{-1} I_\B -A )^{-1} \wt D v = \sum_{j=0}^{+\infty} \zeta^{j+1} E
(j)   [I_\X - \zeta  \wh K (\zeta)]^{-1} D v = \zeta
\wh E (\zeta)  [I_\X - \zeta  \wh K (\zeta)]^{-1} D v , \ \ \zeta \in \DD (e^{\ga}) \setminus \{0\}.
\]

\emph{Now} assume $\B=\B^{\infty,\ga}$.  Generally, (\ref{e Erep}) does not hold in $\B=\B^{\infty,\ga}$.
However, (i) implies that $\vphi_0 := (\zeta^{-1} I_\B -A )^{-1} \wt D v$
 has the form  $\{ \zeta^{-m} \vphi_0^{[0]} \}_{m}$.
So, for $|\zeta|<e^\ga$, we see that $\vphi_0 \in \Bo^{\infty,\ga}$.
Since the representation (\ref{e Erep}) holds for all
$\vphi \in \Bo^{\infty,\ga}$, it holds for $\vphi_0$.
\end{proof}

\begin{proof}[Proof of Theorem \ref{t r ga>0}.]
\emph{Step 1: the proof of (\ref{e rGAS>}).} By Proposition \ref{p r=r new} (i), system (\ref{e Asys}) is UES exactly when (\ref{e Keq}) is UES in $\X$ w.r.t. $\B$.
The boundness of $\IO_K $ follows from those of $\IO_A$,
see (\ref{e IO=}) and the remarks before Theorem \ref{t r>L fo}.
Formula (\ref{e rGAS>}) follows from Theorem \ref{t r fo} (i), Propositions \ref{p
r=r new} (ii) and \ref{p resolvent} (ii).
Note that Proposition \ref{p resolvent} (ii) is applicable since $\zeta$ in (\ref{e rGAS>})
belongs to the unit circle and $e^{\ga}>1$.

\emph{Step 2: the proof of (\ref{e r=}) for the case $p=2$.}
Since $\X$ is a Hilbert space, we see that $\W=\B=\B^{2,\ga}$ is so. Theorem \ref{t r fo}
(ii) and Propositions \ref{p r=r new}-\ref{p resolvent} give
\[ 
\r_\c (K; D , E; \B) =
 \r_\t (K; D , E; \B) =
 \| \IO_K \|_{\ll^2 (\U_2) \to \ll^2 (\U_1)}^{-1} =
\Bigl( \max_{|\zeta|=1} \| \wh E (\zeta)  (I_\X - \zeta  \wh K (\zeta))^{-1} D \|_{\U_2 \to \U_1} \Bigr)\!^{-1}.
\] 

\emph{Step 3: the proof of (\ref{e r=}) for $p \neq 2$.} Let us take $\ga_1 \in (0,\ga)$ and
apply formula (\ref{e r emb}) to the continuous embedding
$\B^{\textstyle 2,\ga_1} \subset \B$. This gives
$
\r_\c (K; D , E; \B^{\textstyle 2,\ga_1} ) \ge \r_\c (K; D , E;
\B) .
$
Formula (\ref{e r=}) for $\B^{\textstyle 2,\ga_1}$ has been proved already on
Step 2 and gives
\[
\r_\c (K; D , E; \B^{2,\ga_1})  =
 \| \IO_K \|_{\ll^2 (\U_2) \to \ll^2 (\U_1)}^{-1} =
 \Bigl( \max_{|\zeta|=1} \| \wh E (\zeta)  [I_\X - \zeta  \wh K (\zeta)]^{-1} D \|_{\U_2 \to \U_1} \Bigr)^{-1}.
\]
These two formulae lead to
\[
\Bigl( \max_{|\zeta|=1} \| \wh E (\zeta)
[I_\X - \zeta  \wh K (\zeta)]^{-1} D \|_{\U_2 \to
 \U_1} \Bigr)^{-1} = \| \IO_K \|_{\ll^2 (\U_2) \to \ll^2 (\U_1)}^{-1}
 \ge \r_\c (K; D , E; \B) .
\]
Combining the latter with formula (\ref{e rGAS>}) obtained on Step 1, we
complete the proof.
\end{proof}

\subsection{Proof of Theorem \ref{t r Bo}: reduction to exponentially fading phase spaces}
\label{ss p t r Bo}

Let system (\ref{e Keq}) be UES in $\X$ w.r.t. $\Bo^{\infty,0}$.
According to Theorem \ref{t UEScr ga=0} (i) $\Leftrightarrow$ (ii),
system (\ref{e Keq}) is UES in $\X$ w.r.t. $\B^{2,\ga}$ for all $\ga \in
(0,\ga_0)$ with certain $\ga_0>0$. This and Theorem \ref{t r ga>0}
imply $\IO_K \in \L \left( \ll^q (\U_2), \ll^q (\U_1)\right) $.

From the assumption that $E (\cdot)$ decays exponentially, we see that there exists $\ga_1>0$ such that the operator $ E $ can
be extended by continuity to the spaces $\B^{2,\ga} $ with $\ga \in
(0,\ga_1)$. We keep the same notation $E$ for these extensions. Put
\[
\ga_2 := \min \{ \ga_0 , \ga_1 \} \ \text{ and } \ r_0 := \Bigl(
\max_{|\zeta|=1} \| \wh E (\zeta)  [I_\X - \zeta  \wh K
(\zeta)]^{-1} D \|_{\U_2 \to \U_1} \Bigr)^{-1} .
\]

Theorem \ref{t r ga>0} and implication (\ref{e r emb}) yield that for all $\ga
\in (0,\ga_2)$
\begin{eqnarray}
 \r_\c (D , E; \Bo^{\infty,0}) & \ge & \r_\c (D , E; \B^{2,\ga})
 =  r_0,
 \label{e r>r>res}
 \\
 \r_\t (D , E; \Bo^{\infty,0}) & \ge & \r_\t (D , E; \B^{2,\ga})  \ge
 \| \IO_K \|_{ \ll^q (\U_2) \to \ll^q (\U_1)}^{-1} .
   \label{e r>r>L-1}
\end{eqnarray}
In particular, $\r_\c (D , E; \B^{2,\ga}) $ does not depend on
the choice of $\ga \in (0,\ga_2)$.

Let us prove that
\[
 \r_\c (D , E; \Bo^{\infty,0}) =  r_0 .
\]
Taking (\ref{e r>r>res}) into account, it is enough to prove that $
\r_\c (D , E; \Bo^{\infty,0}) \le  r_0 $. Assume that the
time-invariant system (Sc) is UES in $\X$ w.r.t. $\Bo^{\infty,0}$.
Then, by Theorem \ref{t UEScr ga=0} (i) $\Leftrightarrow$ (ii) applied to
(Sc), system (Sc) is UES in $\X$ w.r.t. $\B^{2,\ga}$ for certain
$\ga \in (0,\ga_2)$. The definition of $\r_\c (D , E;
\B^{2,\ga}) $ imply that $\| \De \|_{\U_1 \to \U_2} \le \r_\c
(D , E; \B^{2,\ga}) = r_0$ (see anew (\ref{e r>r>res}) and (Sc)).
This imply the desired statement.

Combining the equality $ \r_\c (D , E; \Bo^{\infty,0}) =  r_0 $
with (\ref{e rc>rt>rnt}) and (\ref{e r>r>L-1}) we get (\ref{e rUES>
Bo}).

When $\X$ and $\U_{1,2}$ are Hilbert spaces and $q=2$, Theorem \ref{t r ga>0} implies $r_0 =
\| \IO_K \|_{ \ll^2 (\U_2) \to \ll^2 (\U_1)}^{-1}$. From this and
(\ref{e rUES> Bo}), one can see that (\ref{e rUES> Bo}) holds with
the equalities.

\section{Applications to systems of special types and examples}
\label{s appl}

\subsection{Sufficient conditions for UE stability of time-varying systems}
\label{ss Stab tests}

The following lemma is standard  and can be proved in the same way as in the first order case.

\begin{lemma} \label{l Q=wtQ}
Assume that $\B$  is one of the phase spaces considered in Section \ref{ss Ph sp aux op}.
Let $Q(n) \in \L (\B,\X)$ and $\wt Q (n) \in \L (\B,\X)$ for all $n \in \ZZ^+$.
If $Q(n) = \wt Q(n)$ for $n$ large enough, then the UE stabilities in $\X$ w.r.t. $\B$ of the
systems $x(n+1) = Q(n) x_n$ and $x(n+1) = \wt Q (n) x_n$ are equivalent.
\end{lemma}

Roughly speaking, a modification of a finite number of operators $Q(n)$ in the system $x(n+1) = Q(n) x_n$
does not influence its UE stability.

Let $E(j) \in \L (\X,\U)$ and $\Delta (j) \in \L (\U,\X)$ for all $j \in \ZZ^+$.
Consider the system
\begin{equation} \label{e dec sys}
x (n+1) = \Delta (n) \sum_{j=0}^{+\infty} E (j) x (n-j) , \ \ \ n \ge 0.
\end{equation}

Let us apply Theorem \ref{t r ga>0} to system (\ref{e dec sys}).

\begin{corollary}
Let $\X$ and $\U$ be Hilbert spaces and $\ga >0$.
Let $\| E (j) \|_{\X \to \U} \leq C e^{-\ga j}$ for all $j \in \ZZ^+$ with a certain constant $C$. Then system (\ref{e dec sys})
is UES in $\X$ w.r.t. $B^{1,\ga}$ (and so w.r.t. all $B^{p,\beta}$ with $\beta < \ga$) whenever
\begin{equation}
\limsup_{n \to + \infty} \| \Delta (n) \|_{\U \to \X} <  \frac{1}{ \max_{|\zeta|=1} \| \wh E (\zeta)  \|_{\X \to \U} } .
\end{equation}
\end{corollary}

\begin{proof}
Define an operator $E : \B^{1,\ga} \to \U$ by $E \vphi = \sum_{j=0}^{+\infty} E (j) \vphi^{[-j]} $.
Consider system (\ref{e dec sys}) as a perturbation of (\ref{e Keq}) with $K(j)=\Oz_{\X}$ for all $j$,
 $\U_1 = \U$, $\U_2 = \X$, and $D = I_{\X}$.
Then (\ref{e r=}) implies that (\ref{e dec sys})
is UES in $\X$ w.r.t. $B^{1,\ga}$ whenever
\begin{equation*}
\sup_{n \ge 0} \| \Delta (n) \|_{\U \to \X} <  \frac{1}{ \max_{|\zeta|=1} \| \wh E (\zeta)  \|_{\X \to \U} } .
\end{equation*}
The reference to Lemma \ref{l Q=wtQ} completes the proof.
\end{proof}

For operators $Q (n,j) \in \L (\X)$, $n,j \in \ZZ^+$, consider
the system
\begin{equation} \label{e gen sys}
x (n+1) = \sum_{j=0}^{+\infty} Q (n,j) x (n-j) , \ \ \ n \ge 0.
\end{equation}

\begin{corollary} \label{c Q p b}
Let $\X$ be a Banach space, $\beta >0$, and $1 \le p,p' \le \infty$ be such that $1/p+1/p' =1$.
Assume that for each $n \ge 0$ the sequence $\{ \| e^{j \beta} Q (n,j) \|_{\X \to \X} \}_{j=0}^{+\infty}$
belongs to $\ll^{p'}$. Then system (\ref{e gen sys}) is UES in $\X$ w.r.t. $\B^{p,\beta}$
whenever
\begin{eqnarray} \label{e Qcond p>1}
\limsup_{n\to \infty} \sum_{j=0}^{+\infty} \| e^{j\beta} Q (n,j) \|_{\X \to \X}^{p'} & < & \left( 1-e^{-p\beta} \right)^{1/(p-1)}
 \text{ in the case } 1<p\leq \infty,\\
\text{and } \limsup\limits_{n \to \infty} \ \sup_{j \ge 0} \| e^{j\beta} Q(n,j) \|_{\X \to \X} & < & 1-e^{-\beta} \qquad \text{ in the case } p= 1,
\label{e Qcond p=1}
\end{eqnarray}
where  $e^{-p\beta} $ and $1/(p-1)$ have to be understood as zero when
$p=\infty$.
\end{corollary}

\begin{proof}
Define operators $N(n) : \B^{p,\beta} \to \X$ by $N(n) \vphi = \sum_{j=0}^{+\infty} Q (n,j) \vphi^{[-j]} $
and consider system (\ref{e gen sys}) as an unstructured perturbation of (\ref{e Keq}) with $K(j)=\Oz_{\X}$, $j \ge 0$.
The norm of the unstructured input-state operator $\IS_K $ equals 1 in each of $\ll^s$-spaces.
By Corollary \ref{c unstr}, system (\ref{e gen sys}) is UES in $\X$ w.r.t. $\B^{p,\beta}$ whenever
$\sup_n \| N (n) \|_{\B^{p,\beta} \to \X} < \left( 1-e^{-p\beta} \right)^{1/p}$.
Since
\begin{eqnarray}
\| N (n) \|_{\B^{p,\beta} \to \X} & \le  & \left( \sum_{j=0}^{+\infty} \| e^{j\beta} Q (n,j) \|_{\X \to \X}^{p'}  \right)^{1/p'}
 \text{ when } 1<p\leq \infty,\\
\text{and } \ \ \| N (n) \|_{\B^{1,\beta} \to \X} & \le & \sup_{j \ge 0} \| e^{j\beta} Q(n,j) \|_{\X \to \X}
\qquad \text{ when } p= 1,
\end{eqnarray}
we see that (\ref{e gen sys}) is UES in $\X$ w.r.t. $\B^{p,\beta}$ if
\begin{eqnarray}
\sup_n \sum_{j=0}^{+\infty} \| e^{j\beta} Q (n,j) \|_{\X \to \X}^{p'} & < & \left( 1-e^{-p\beta} \right)^{p'/p} = \left( 1-e^{-p\beta} \right)^{1/(p-1)}
\ \text{for }  1<p\leq \infty,\\
\text{and } \ \  \sup_{n,j} \| e^{j\beta} Q(n,j) \|_{\X \to \X} &  <  & 1-e^{-\beta}  \ \ \text{for }  p= 1.
\end{eqnarray}
Lemma \ref{l Q=wtQ} completes the proof.
\end{proof}

Now Proposition \ref{p st emb} makes it possible to give sufficient conditions of UE stability w.r.t. the non-fading phase spaces $\B^{\infty,0}$ and
$\Bo^{\infty,0}$.

\begin{corollary}
Let $\X$ be a Banach space, $0 \leq \ga < \alpha $, and $1 \le q \le \infty$.
Assume that for each $n \ge 0$ the operator $Q(n) = \sum_{j=0}^{+\infty} Q (n,j) \vphi^{[-j]} $ is bounded
in $\B^{q,\ga}$ (in $\Bo^{\infty,\ga}$). Here
the convergence of the infinite sum is understood in the sense of the norm topology of $\X$.

Assume that, for $n$ large enough, there exist constants $C(n)$ such that
$Q(n,j) \leq C(n) e^{-j \alpha}$ for all $j \ge 0 $. Then each of conditions (\ref{e Qcond p>1}), (\ref{e Qcond p=1})
with arbitrary $\beta \in (\ga,\alpha)$ and arbitrary $p$ in the range $1 \le p \le \infty$ implies the UE stability of system (\ref{e gen sys}) in $\X$
w.r.t. $\B^{q,\ga}$ (resp., w.r.t. $\Bo^{\infty,\ga}$).
\end{corollary}

\begin{proof}
According to the assumptions, it is possible to modify $Q(n,j)$ for $0 \le n \le n_0 < \infty$ such that
the modified system is defined on each of phase spaces $\B^{p,\beta}$ with $\beta < \alpha$.
The UE stability of the initial and the modified system in $\X$ w.r.t. $\B^{q,\ga}$ are equivalent due to Lemma \ref{l Q=wtQ}.
Applying Corollary \ref{c Q p b} to the modified system, we see that it is UES in the $\B^{p,\beta}$ settings.
For $\beta > \ga$ Proposition \ref{p st emb} implies that both the modified and the original system is UES in $\X$
w.r.t. $\B^{q,\ga}$. For the case of $\Bo^{\infty,\ga}$, the proof is the same.
\end{proof}

The condition (\ref{e Qcond p>1}) for $p=\infty$ and $\beta>0$ improves
the sufficient condition for UE stability in the resolvent matrix sense given by \cite[formula (3.1)]{CKRV00}.

\subsection{A delayed feedback scheme} \label{ss Del feedback}

Consider another feedback scheme given by Fig.\ref{fig2}.
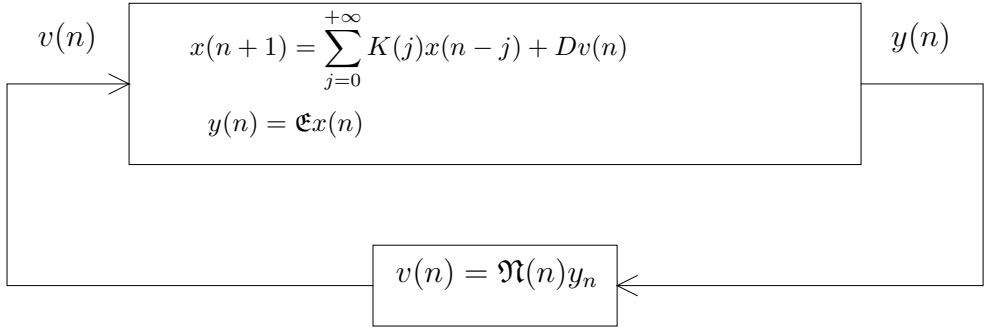
\begin{figure}[ht]
\centering
\setlength{\unitlength}{0.0007in}
\begingroup\makeatletter\ifx\SetFigFont\undefined%
\gdef\SetFigFont#1#2#3#4#5{%
  \reset@font\fontsize{#1}{#2pt}%
  \fontfamily{#3}\fontseries{#4}\fontshape{#5}%
  \selectfont}%
\fi\endgroup%
{\renewcommand{\dashlinestretch}{30}
\begin{picture}(7224,2439)(0,-10)
\path(912,2412)(6312,2412)(6312,1212)
    (912,1212)(912,2412)
\path(6312,1812)(7212,1812)(7212,312)(4512,312)
\path(2712,612)(2712,12)(4512,12)
    (4512,612)(2712,612)
\path(2712,312)(12,312)(12,1812)
    (912,1812)(762,1887)
\path(762,1737)(912,1812)
\path(4662,387)(4512,312)(4662,237)
\put(237,2037){\makebox(0,0)[lb]{{\SetFigFont{12}{12}{\rmdefault}{\mddefault}{\updefault}$v(n)$}}}
\put(6537,2037){\makebox(0,0)[lb]{{\SetFigFont{12}{12}{\rmdefault}{\mddefault}{\updefault}$y(n)$}}}
\put(2880,280){\makebox(0,0)[lb]{{\SetFigFont{12}{12}{\rmdefault}{\mddefault}{\updefault}$v(n)=\N(n) y_n$}}}
\put(1362,1750){\makebox(0,0)[lb]{{\SetFigFont{10}{10}{\rmdefault}{\mddefault}
{\updefault}${\displaystyle x(n+1)=\sum_{j=0}^{+\infty} K(j)x(n-j)+Dv(n)}$}}}
\put(1362,1412){\makebox(0,0)[lb]{{\SetFigFont{10}{10}{\rmdefault}{\mddefault}{\updefault}$~~y(n)
= \mathfrak{E} x (n)$}}}
\end{picture}
}
\caption{Delayed feedback.}
\label{fig2}
\end{figure}
Here $y (n) \in \Ufd_1$  is an output depending now only on the state $x (n)$ the system, but
the input $v(n) \in \Ufd_2 $ is connected with the output by $v(n) = \N (n) y_n$
and so depends on the prehistory of the output.
Here the Banach space $\Ufd_2$ ($\Ufd_1$) is the input  (resp., output) space.

In this section, we will use the space $\B^{p,\ga} ( \Ufd_1)$, which is defined similar to $\B^{p,\ga} $, but with $\Ufd_1$
instead of $\X$ (so that $\B^{p,\ga} =\B^{p,\ga} (\X)$).
Suppose that $\E \in \L (\X, \Ufd_1)$ and that the prehistory of the output
$y_n = \{ y_n^{[m]} \}_{m=-\infty}^{0} := \{ y (n+m) \}_{m=-\infty}^{0}$
 belongs to $\B^{p,\ga} ( \Ufd_1)$. Then it is natural to assume that \emph{unknown feedback operators} $\N (n)$ map $\B^{p,\ga} (\Ufd_1)$
to $\Ufd_2$.
One can define corresponding
stability radii similar to that of Section \ref{ss r def}.

However, we do not want to introduce a new notation because corresponding perturbed systems
\emph{can be considered as particular cases} of systems (Sc)-(Snt). For this purpose, consider
the diagonal operator
\[
M_\E : \B^{p,\ga} \to \B^{p,\ga} (\Ufd_1) \ \text{ defined by } \
(M_\E \vphi)^{[m]} = \E \vphi^{[m]} , \ \ m \in \ZZ^-,
\]
and put
\[
\U_1 = \B^{p,\ga} (\Ufd_1), \ \ \ \U_2 = \Ufd_2, \ \ \ \text{ and } E = M_\E.
\]
Then the following perturbed system can be associated with Fig.\ref{fig2}:
\begin{equation} \label{e per sys del f}
  x(n+1)=\sum_{j=0}^{+\infty} \ K(j) \ x(n-j) + D \ \N (n) \ M_\E x_n .
 \end{equation}
So $\r_i ( D , M_\E; \B)$, $i= \c, \t$, are \emph{the stability radii for the delayed feedback scheme.}

\begin{remark} \label{r rad MM08}
In the case when $K(j)$ are positive compact operators on a complex Banach lattice $\X$, and $D$, $E$,
$\N (n)$ satisfy certain additional assumptions, a radius of asymptotic stability defined similar to
$\r_\c ( D , M_\E; \Bo^{\infty,0})$ was considered in \cite[Sect.4]{MN08}.
\end{remark}

The input-output operator $\IOfd_{K}: \Ssp_+ (\U_2) \to \Ssp_+ (\U_1)$
associated with Fig.\ref{fig2} is defined by $\IOfd_K : v (\cdot) \to y (\cdot)$,  where $y
(n) = \E x (n)$, $n \ge 0$, and $x (\cdot)$ is the solution to the
system (\ref{e Keq in}).

Note that the operator $\IO_K$ associated with (\ref{e per sys del f})
differs from $\IOfd_{K}$, though they are obviously connected by
\begin{equation*} \label{e IO=(I0fd)}
(\IO_K v) (n) = M_\E x_n = \{ \dots, \E x(n-1) , \E x(n) \} = \{  \dots , (\IOfd_{K} v) (n-1) , (\IOfd_{K} v) (n) \} .
\end{equation*}

\begin{corollary} \label{c del feedback ga>0}
 Let $\ga >0$ and $1 \le p \le \infty$. Let $\B = \B^{p,\ga}$ or $\B=\Bo^{\infty,\ga}$
 (in the latter case $p$ is assumed to be equal to $\infty$ and $M_\E$ is assumed to be restricted to $\Bo^{\infty,\ga}$).
 Let (\ref{e Keq}) be UES in $\X$ w.r.t. $\B$. Then
\begin{multline}
\left( 1 - e^{-p\ga} \right)^{1/p} \Bigl( \max_{|\zeta|=1}
 \| \E [I_\X - \zeta  \wh K (\zeta)]^{-1} D \|_{\Ufd_2 \to \Ufd_1} \Bigr)^{-1}
= \r_\c (D , M_\E; \B)  \ge  \r_{\t} (D , M_\E; \B) \ge \\ \ge
\left( 1-e^{-p\ga} \right)^{1/p}
\| \IOfd_K \|_{ \ll^p (\Ufd_2) \to \ll^p ( \Ufd_1) }^{-1} > 0,
\label{e rGAS> fd}
\end{multline}
where  $e^{-p\ga} $ and $1/p$ have to be understood as zero when $p=\infty$.

If $p=2$ and $\X$, $\Ufd_1$, $\Ufd_2$ are Hilbert spaces, the equalities hold
in (\ref{e rGAS> fd}).
\end{corollary}

\begin{proof}
It is enough to apply Theorem \ref{t r ga>0} and to perform calculations similar to that of Section \ref{s unstr}.
The first equality in (\ref{e rGAS> fd}) requires additional explanations.
The discrete function $M_\E ( \cdot)$ constructed by the operator $M_\E$ (see Section \ref{ss Ph sp aux op})
is given by
\[
M_\E (n) = M_\E P_{-n}^\T = P_{-n}^\T \E .
\]
Here we extended the definition of $P_{-n}^\T$ given in Section \ref{ss Ph sp aux op}
to the space $\B^{p,\ga} ( \Ufd_1)$.
So
\[
\left( \wh M_\E (\zeta) \psi \right)^{[m]} =
\zeta^{-m}  \E \psi , \ \  m \in \ZZ^- .
\]
When $|\zeta| =1 $, $v \in \Ufd_2$, and $p<\infty$,
\begin{multline*}
| \wh M_\E (\zeta)  [I_\X - \zeta  \wh K (\zeta)]^{-1} D  v |_{\U_1}^p  =
| \wh M_\E (\zeta)  [I_\X - \zeta  \wh K (\zeta)]^{-1} D  v |_{\B^{p,\ga} (\Ufd_1)}^p =  \\
\qquad \sum_{m=-\infty}^{0} e^{pm\ga} | \E [I_\X - \zeta  \wh K (\zeta)]^{-1} D  v |_{\Ufd_1}^p =
(1-e^{-p\ga})^{-1}  | \E [I_\X - \zeta  \wh K (\zeta)]^{-1} D  v |_{\Ufd_1}^p .
\end{multline*}
This gives the desired equality (with standard changes for $p=\infty$).

To get the last inequality in (\ref{e rGAS> fd}), we use the formula
\begin{equation*} \label{e IO=IOfd}
\| \IO_K \|_{ \ll^p (\Ufd_2) \to \ll^p \left( \B^{p,\ga} (\Ufd_1) \right)} = \left( 1-e^{-p\ga} \right)^{-1/p}
\| \IOfd_K \|_{ \ll^p (\Ufd_2) \to \ll^p ( \Ufd_1) } ,
\end{equation*}
which can be obtained in the same way as (\ref{e IO=IS}).
\end{proof}

\begin{corollary} \label{c r=0 fd}
If $\E \neq 0_{\X \to \Ufd_1}$ and $D \neq 0_{\Ufd_2 \to \X}$, then $ \r_i (D , M_\E; \Bo^{\infty,0})  = 0 ,
i=\c,\t$ (here $M_\E$ is assumed to be restricted to $\Bo^{\infty,0}$).
\end{corollary}

\begin{proof}
In this case, $\B = \Bo^{\infty,0}$. The function $M_\E ( \cdot)$ (which corresponds to $E (\cdot)$ of Proposition \ref{p r=0})
does not decay exponentially since
$ \| M_\E ( n ) \|_{\X \to \U_1} = \| P_{-n}^\T \E \|_{\X \to \B^{\infty,0} (\Ufd_1)}
 = \| \E \|_{\X \to \Ufd_1} $ is a positive constant. Proposition \ref{p r=0}
completes the proof.
\end{proof}

\subsection{An example of a perturbed non-positive system}
\label{ss ex nonpos}

Take $\X =\CC$ and consider the systems
\begin{equation} \label{e ex - 12 j per}
x(n+1) = - \sum_{j=0}^\infty 2^{-j} x (n-j) + \sum_{j=0}^{+\infty} \Delta (n,j) x (n-j) , \ \ n \geq 0 ,
\end{equation}
with uncertain complex coefficients $\Delta (n,j) \in \CC$, $n,j \in \ZZ^+$.
The problem is to find conditions on $\Delta (n,j)$ that ensure the UE stability of these systems
in $\X$ with respect to a certain phase space (by Remark \ref{r EM=SwrtBo}, such conditions guarantee also
the UE stability in the resolvent matrix sense).

We consider systems (\ref{e ex - 12 j per}) as perturbations of the convolution system
\begin{equation} \label{e ex - 12 j}
x(n+1) = - \sum_{j=0}^\infty 2^{-j} x (n-j) .
\end{equation}

\emph{First, consider stability properties of system (\ref{e ex - 12 j}).}
It is a  system of the type (\ref{e Keq}) with $K(j)= - 2^{-j}$.
The Z-transform of $K (\cdot)$ equals $\wh K (\zeta) = \frac{2}{\zeta-2}$.
The radius of convergence of $\wh K (\zeta)$ equals $R [\wh K] = 2$.
System (\ref{e ex - 12 j}) is defined on the phase spaces
\begin{equation} \label{ex - 12 j ph sp}
\B^{p,\ga} \text{ for all } -\infty < \ga < \ln 2, \ 1 \leq p \leq \infty,  \text{ and also on } \ \B^{1,\ln 2}.
\end{equation}

Recall that $X_K (\cdot)$ is the convolution kernel corresponding to the unstructured input-state operator $\IS_K $,
see (\ref{e IS K conv}), and that $X_K (\cdot)$ is connected with
the resolvent matrix $X_K (\cdot, \cdot)$ of (\ref{e ex - 12 j}) by $X_K (n,j) = X_K (n-j)$.
According to Lemma  \ref{l whK=whXK new}, the Z-transform of the function $X_K (\cdot)$ equals
\begin{equation} \label{e ex - 12 j wh X}
\wh X_K (\zeta) = [1-\zeta\wh K(\zeta)]^{-1} = \frac{2- \zeta}{2+\zeta}.
\end{equation}

Recovering the function $X_K (\cdot)$ from its Z-transform, one gets
$X(0) =1$ and $ X (j) = - (-\frac12)^{j-1} $ for $j \in \NN$. By (\ref{e IS K conv}), the explicit form of the
the unstructured input-state operator $\IS_K $ is
\begin{equation} \label{e ex -12 IS}
(\IS_K f) (n) = f(n-1) - \sum_{j=2}^{n} \left(-\frac12 \right)^{j-2} f (n-j) , \ \ \ (\IS_K f) (0) = 0 .
\end{equation}

We see that $X_K (\cdot)$ decays exponentially. In other words, \emph{
system (\ref{e ex - 12 j}) is UES in the resolvent matrix sense.}
By Theorem \ref{t UEScr ga>0} and Proposition \ref{p st emb}, \emph{system (\ref{e ex - 12 j}) is UES
in $\X$ w.r.t. each of the phase spaces of (\ref{ex - 12 j ph sp}).
For $\ga < \ln 2$, system (\ref{e ex - 12 j}) is also UES in $\X$ w.r.t. the spaces $\Bo^{\infty,\ga}$}, which are
isometrically embedded in $\B^{\infty,\ga}$.


\emph{Let us study stability radii of (\ref{e ex - 12 j}) under unstructured perturbations }
in a phase space $\B=\B^{p,\ga}$ assuming that either $0 < \ga < \ln 2$, $1 \leq p \leq \infty$, or
$ \ga = \ln 2$, $p=1$. We want to use the settings of Section \ref{s unstr} to calculate (or estimate)
the stability radii $\r_\c (K; \B)$ and $\r_\t (K; \B)$.

Clearly,
\[
\max_{|\zeta|=1} | [1 - \zeta  \wh K (\zeta)]^{-1} | = \max_{|\zeta|=1} \left| \frac{2- \zeta}{2+\zeta} \right| =
3 .
\]
By Corollary \ref{c unstr},
\[
\r_\c (K; \B^{p,\ga}) = \frac{(1-e^{-p\ga})^{1/p}}{3} .
\]

Time-varying stability radii $\r_\t$ can be easily calculated when $p=1,2,\infty$:
\begin{equation} \label{e ex -12 rt}
\frac{(1-e^{-p\ga})^{1/p}}{3} =  \r_\t (K; \B^{p,\ga}) =
(1-e^{-p\ga})^{1/p} \ \| \IS_K \|_{ \ll^p (\X) \to \ll^p ( \X) }^{-1} \ , \qquad p=1,2,\infty .
\end{equation}
Indeed, for $p=2$ this equality is provided immediately by Corollary \ref{c unstr}.
When $p=1$ or $p=\infty$, the norms of the unstructured input-state operator
can be calculated via (\ref{e ex -12 IS}):
\[
\| \IS_K \|_{ \ll^1 (\X) \to \ll^1 ( \X) } = \| \IS_K \|_{ \ll^\infty (\X) \to \ll^\infty ( \X) } = 3 .
\]
(The supremum of norms of $\IS_K f$ over the corresponding unit balls are
archived for suitable $f(\cdot)$ with alternating signs of $f(n)$.)
Hence,  $\max_{|\zeta|=1} | [1 - \zeta  \wh K (\zeta)]^{-1} |  = \| \IS_K \|_{ \ll^p (\X) \to \ll^p ( \X) }$,
 and therefore, (\ref{e r> unstr}) holds with equalities. This proves (\ref{e ex -12 rt}).

\emph{Finally, we derive stability conditions for (\ref{e ex - 12 j per}) in terms of coefficients}
using the obtained stability radii. To write system (\ref{e ex - 12 j per}) in the form (\ref{e unstr}),
we define the operators (actually, the functionals) $N(n):\B^{p,\ga} \to \CC$ by
$N(n) \vphi =  \sum_{j=0}^{+\infty} \ \Delta(n,j) \ \vphi^{[-j]}$. Then,
\begin{eqnarray}
\| N (n) \|_{\B^{p,\ga} \to \CC} & =  & \left( \sum_{j=0}^{+\infty} | \Delta(n,j) e^{j\ga}|^{p'}  \right)^{1/p'}
 \text{ when } 1<p\leq \infty,\\
\text{and } \| N (n) \|_{\B^{1,\ga} \to \CC} & = & \sup_{j \ge 0} | \Delta(n,j) e^{j\ga}|
\qquad \text{ when } p= 1,
\end{eqnarray}
where $p'$ is the H\"{o}lder conjugate of $p$, $1/p' + 1/p =1 $.

Combining the definition of $\r_\t (K; \B^{p,\beta})$ with (\ref{e ex -12 rt}), we see that
\emph{(\ref{e ex - 12 j per}) is UES in $\X$ w.r.t. $\B^{p,\beta}$ in each of the following cases}:\\
$p=1$, $0<\beta \leq \ln 2$, and
\begin{equation} 
\sup_{j \ge 0 } | \Delta(n,j) e^{j\beta}| < \frac{1-e^{-\beta}}{3} \ \  \text{ for all } n \ge 0 ;
\tag{N1}
\end{equation}
$p=2$, $0<\beta < \ln 2$, and
\begin{equation} 
\sum_{j=0}^{+\infty} | \Delta(n,j) e^{j\beta}|^2 < \frac{1-e^{-2\beta}}{9} \ \ \text{ for all } n \ge 0 ;
\tag{N2}
\end{equation}
$p=\infty$, $0<\beta < \ln 2$, and
\begin{equation} 
\sum_{j=0}^{+\infty} | \Delta(n,j) e^{j\beta}| < \frac{1}{3} \ \ \text{ for all } n \ge 0  .
\tag{N$\infty$}
\end{equation}

The continuous embedding $\B^{1,\delta} \subset \B^{2,\delta} \subset \B^{\infty,\delta}$ and
Proposition \ref{p st emb} imply also that, \emph{in the case $p=1$, $0<\beta < \ln 2$,
(\ref{e ex - 12 j per}) is UES in $\X$ w.r.t. $\B^{1,\beta}$ whenever any of the conditions
(N2) or (N$\infty$) is satisfied.}

Note that \emph{conditions (N1), (N2), and (N$\infty$) are independent,}
i.e., none of them implies another one.

Similarly, \emph{in the case $p=2$, $0<\beta < \ln 2$,
(\ref{e ex - 12 j per}) is UES in $\X$ w.r.t. $\B^{2,\beta}$ whenever (N$\infty$) is satisfied.}

The unstructured stability radii corresponding to $\B = \B^{\infty,0}$ and $\B = \Bo^{\infty,0}$ do not produce
stability tests since these radii are equal to $0$, see Corollary \ref{c unstr Bo}.
However, the continuous embedding argument allows one to obtain sufficient conditions of UE stability in $\X$ w.r.t.
$\B^{\infty,0}$ and $\Bo^{\infty,0}$, as well as w.r.t. $\B^{p,\ga}$ with $p \neq 1,2,\infty$.
In fact, embedding (\ref{e emB}), Proposition \ref{p st emb}, and the above results yield the following conditions
(since the produced conditions for the phase spaces $\B^{\infty,0}$ and $\Bo^{\infty,0}$ coincide, below
we give only $\B^{\infty,0}$ version).

\begin{proposition}
Let $0 \le \ga < \ln 2$. System (\ref{e ex - 12 j per}) is UES in $\X$ w.r.t. $\B^{p,\ga}$
if the condition (N1) is fulfilled for a certain $\beta \in (\ga,\ln 2]$ or if any of the conditions
(N2), (N$\infty$) is fulfilled for a certain $\beta \in (\ga,\ln 2)$.
\end{proposition}

These scales of stability tests have the following additional properties:
\begin{itemize}
\item[(i)] as before, none of the above conditions imply another one (even produced by a different $\beta$),
\item[(ii)] the constants in the right sides of (N1), (N2), and (N$\infty$) are sharp, more precisely,
for each of the conditions (N1), (N2), and (N$\infty$), there exist $\Delta (n,j) $  such that the equality holds in
the corresponding formula, but (\ref{e ex - 12 j per}) is not
UES in $\X$ w.r.t. any of phase spaces $\B^{p,\ga}$.
\item[(iii)] using Lemma \ref{l Q=wtQ}, the requirement 'for all $n \ge 0$' in (N1), (N2), and (N$\infty$) can be weakened to
'for all n large enough' .
\end{itemize}

Statement (i) can be easily seen by direct examination.

\emph{Let us prove (ii) for the case of (N$\infty$).}
Taking $\Delta (n,0) = -1/3$ for all $n \ge 0$, and $\Delta (n,j) = 0$ for all $j \ge 1$ and $n \ge 0$,
we see by straightforward calculations that the equality holds in (N$\infty$), that the convolution (time-invariant)
system (\ref{e ex - 12 j per})
is defined for all phase spaces of (\ref{ex - 12 j ph sp}), and that for system (\ref{e ex - 12 j per})
the condition (ii) of Theorem \ref{t UEScr ga>0} is not valid when $\zeta = -1 $.
Hence, (\ref{e ex - 12 j per}) is not UES in the resolvent matrix sense.
\emph{The equality holds in (N1)} if $\Delta (n,j) = - \frac{1-e^{-\beta}}{3} (-1)^{j} e^{-\beta j} $ for all $n$.
Though the corresponding convolution  system (\ref{e ex - 12 j per}) is defined on $\B^{1,\beta}$
and all the embedded phase spaces, it is not UES
in the resolvent matrix sense. Indeed, condition (ii) of Theorem \ref{t UEScr ga>0} is not valid again for $\zeta = -1 $.
\emph{Taking $\Delta (n,j) = - \frac{1-e^{-2\beta}}{3} (-1)^{j} e^{-2\beta j} $, we see that the equality holds in (N2),}
but the system it is not UES in the resolvent matrix sense by the same reason as before.




\begin{thebibliography}{99}

\bibitem{KC-VT-M03} V.B. Kolmanovskii, E. Castellanos-Velasco, J. A. Torres-Mu\~{n}oz,
A survey: stability and boundedness of Volterra difference equations. Nonlinear Anal. 53 (7-8)(2003) 861--928.

\bibitem{F68}
W. Feller, An Introduction to Probability Theory and its Applications, Vol. I, third edition, John Wiley \& Sons,
New York-London-Sydney, 1968.

\bibitem{CJRV91}
M.R. Crisci, Z. Jackiewicz, E. Russo, A. Vecchio, Stability analysis of discrete recurrence equations of
Volterra type with degenerate kernels, J. Math. Anal. Appl. 162 (1)(1991) 49--62.

\bibitem{FMN04_JMJ}
T. Furumochi, S. Murakami, Y. Nagabuchi,
Volterra difference equations on a Banach space and abstract differential equations with piecewise continuous delays,
Japan. J. Math. (N.S.) 30 (2)(2004) 387--412.

\bibitem{FMN04_JDEA}
T. Furumochi, S. Murakami, Y. Nagabuchi, A generalization of Wiener's lemma and its application
to Volterra difference equations on a Banach space,
J. Difference Equ. Appl. 10 (13-15)(2004) 1201--1214.

\bibitem{P80}
K.M. Przy{\l}uski,
The Lyapunov equation and the problem of stability for linear bounded discrete-time systems in Hilbert space,
Appl. Math. Optim. 6 (2)(1980) 97--112.

\bibitem{M97}
S. Murakami, Representation of solutions of linear functional difference equations
in phase space, Nonlinear Anal. T.M.A. 30 (2)(1997) 1153--1164.

\bibitem{CP00}
C. Cuevas, M. Pinto, Asymptotic behavior in Volterra difference systems with unbounded delay. Fixed
point theory with applications in nonlinear analysis, J. Comput. Appl. Math. 113 (1-2)(2000) 217--225.

\bibitem{BK10}
E. Braverman and I. Karabash, Bohl--Perron-type stability
theorems for linear difference equations with infinite delay,  J. Difference Equ. Appl. (2011),
DOI:10.1080/10236198.2010.531276.

\bibitem{HMN91}
Y. Hino, S. Murakami, T. Naito, Functional Differential
Equations with Infinite Delay, Lecture Notes in Mathematics,  1473,
Springer-Verlag,  Berlin, 1991.

\bibitem{CKRV00}
M.R. Crisci, V.B. Kolmanovskii, E. Russo, A. Vecchio,
On the exponential stability of discrete Volterra systems.  J. Differ. Equations Appl.  6 (6)(2000),
667--680.

\bibitem{EM96}
S. Elaydi, S. Murakami,
Asymptotic stability versus exponential stability in linear Volterra
difference equations of convolution type,  J. Difference Equ. Appl. 2 (4)(1996) 401--410.

\bibitem{PH93}
G. Pappas, D. Hinrichsen,
Robust stability of linear systems described by higher order dynamic equations, IEEE Trans. Autom.
Control 38 (9)(1993) 1430--1435.

\bibitem{WH94}
F. Wirth, D. Hinrichsen,
On stability radii of infinite-dimensional time-varying discrete-time systems,
IMA J. Math. Control Inform. 11 (3)(1994) 253--276.

\bibitem{SB03}
Y. Song, C.T.H. Baker,
Perturbation theory for discrete Volterra equations,
J. Difference Equ. Appl. 9 (10)(2003) 969--987.

\bibitem{SB04}
Y. Song, C.T.H. Baker,
Perturbations of Volterra difference equations,
J. Difference Equ. Appl. 10 (4)(2004), 379--397.

\bibitem{K06}
V.B. Kolmanovskii,
Robust stability of Volterra discrete equations under perturbations of their kernels,
Dynam. Systems Appl. 15 (3-4)(2006) 333--342.

\bibitem{MN08}
S. Murakami, Y. Nagabuchi, Uniform asymptotic stability and robust stability
for positive linear Volterra difference equations in Banach lattices,
Adv. Difference Equ. 2008, Art. ID 598964, 15 pp.

\bibitem{NNShM09}
P.H.A. Ngoc, T. Naito, J.S. Shin, S. Murakami,
Stability and robust stability of positive linear Volterra difference equations,
Internat. J. Robust Nonlinear Control 19 (5)(2009) 552--568.

\bibitem{E09}
S. Elaydi,
Stability and asymptoticity of Volterra difference equations: a progress report,
J. Comput. Appl. Math. 228 (2)(2009), 504--513.

\bibitem{K81}
E.R. Kanasewich,  Time Sequence Analysis in Geophysics, University of Alberta, 1981.

\bibitem{BB06}
L. Berezansky and E. Braverman,
On exponential dichotomy for linear difference equations with bounded and unbounded delay,
Differential \& Difference Equations and Applications,  169--178, Hindawi Publ. Corp., New York, 2006.

\bibitem{DSh58}
N. Dunford, J.T. Schwartz, Linear Operators. I. General Theory.
With the assistance of W.G. Bade and R.G. Bartle, Interscience Publishers, New York - London, 1958.

\bibitem{S50}
G. Sirvint, Weak compactness in Banach spaces, Studia Math. 11 (1950) 71--94.

\bibitem{MN05}
S. Murakami, Y. Nagabuchi,
Stability properties and asymptotic almost periodicity for linear
Volterra difference equations in a Banach space,
Japan. J. Math. (N.S.) 31 (2)(2005) 193--223.

\bibitem{P87}
K.M. Przy{\l}uski,
Remarks on $\ll^p$-input bounded-state stability of
linear controllable infinite-dimensional
systems, Syst. Control Lett. 9 (1)(1987) 73--77.

\bibitem{SS04}
B. Sasu, A.L. Sasu, Stability and stabilizability for linear systems of difference
equations, J. Differ. Equations Appl. 10 (12)(2004) 1085--1105.

\bibitem{AVM96}
B. Aulbach, N. Van Minh, The concept of spectral dichotomy for
linear difference equations. II,  J. Differ. Equations Appl.
2 (3)(1996) 251--262.

\bibitem{PR84}
K.M. Przy{\l}uski, S. Rolewicz, On stability of
linear time-varying infinite-dimensional discrete-time systems,
Systems Control Lett. 4 (5)(1984) 307--315.


\end{thebibliography}
\end{document}